\documentclass{article}
\usepackage[utf8]{inputenc}
\usepackage{bbm,tikz}
\usepackage[numbers]{natbib}

\RequirePackage{amsthm,amsmath,amsfonts,amssymb}
\usepackage{amssymb}
\usepackage[utf8]{inputenc}
\usepackage{bbm,tikz}
\RequirePackage{amsthm,amsmath,amsfonts,amssymb}
\RequirePackage{graphicx,url,color}
\usepackage{comment}
\usepackage{soul}
\usepackage{algpseudocode}
\usepackage{algorithm,color}
\usepackage{enumitem,amsmath}   
\usepackage{dsfont}
\usepackage{calc}

\newtheorem{theorem}{Theorem}[section]
\newtheorem{thm}{Theorem A.\ignorespaces}

\newtheorem{definition}[theorem]{Definition}

\newtheorem*{theorem*}{Theorem}
\newtheorem{example}[theorem]{Example}
\newtheorem{exam}[thm]{Example A.\ignorespaces}

\newtheorem{problem}[theorem]{Problem}
\newtheorem{proposition}[theorem]{Proposition}
\newtheorem{remark}[theorem]{Remark}

\newcommand{\RR}{\mathbb{R}}
\newcommand{\R}{\mathbb{R}}

\newcommand{\T}{\mathbb R^e \!/\mathbb R {\bf 1}}

\newcommand{\tconv}{\text{tconv\,}}

\begin{document}

\title{Maximum Inscribed and Minimum Enclosing Tropical Balls of Tropical Polytopes and Applications to Volume Estimation and Uniform Sampling}
\author{David Barnhill \and Ruriko Yoshida  \and  Keiji Miura}
\date{}

\maketitle

\begin{abstract}
We consider a minimum enclosing and maximum inscribed tropical balls for any given tropical polytope over the tropical projective torus in terms of the tropical metric with the max-plus algebra.  We show that we can obtain such tropical balls via linear programming.  Then we apply minimum enclosing and maximum inscribed tropical balls of any given tropical polytope to estimate the volume of and sample uniformly from the tropical polytope.
\end{abstract}

%% main text
\section{Introduction}
\begin{comment}
\begin{enumerate}
    \item Show how to compute the maximum inscribed tropical ball and minimum enclosing tropical ball of a tropical polytope.
    \item Application to the volume estimation of a tropical polytope.
    \item Application to uniform sampling via HAR sampler from a tropical polytope.
\end{enumerate}
\end{comment}

Tropical polytopes \citep{MS, joswigBook} are gaining popularity as a tropical counterpart of classical polytopes.
It is expected that tropical convex geometry has potential to solve many problems \cite{YZZ, YTMM, 10.1093/bioinformatics/btaa564, Trop_KDE, Trop_HAR}, just as classical convex geometry has many applications.
In fact, tropical polytopes over the tropical projective space have been studied thoroughly (for examples, see \cite{MS,joswigBook,TRAN20171} and references within) and have been applied in many areas, such as statistics, optimization, and phylogenomics \cite{MS,YZZ,YTMM,10.1093/bioinformatics/btaa564,TRAN20171,DS}. 
%In fact, classical convex geometry is very useful because it has rich contents such as minimum enclosing balls, volume estimation, and uniform sampling, to mention a few.
%However, the field of tropical polytopes is not yet fully investigated and it lacks some counterpart fields for classical convex polytopes.
%Thus it is essential to study basic concepts of polytopes in a tropical setting.

In polyhedral geometry, the minimum enclosing ball (or sphere) of a given polytope is the smallest ball in terms of the Euclidean metric, $l_2$-norm, containing the given polytope. The  maximum inscribed ball of a polytope is the largest ball in terms of $l_2$-norm that is contained in the polytope.  Computing the minimum enclosing ball and maximum inscribed ball of a polytope is a challenging problem \cite{konno,Murty}. The maximum inscribed ball of a given polytope is closely related to the interior point method for linear programming, and both the minimum enclosing ball and the  maximum inscribed ball of a given polytope have been applied to solving linear programming problems \cite{Murty}, to computing support vector machines \cite{C, CHW}, and to estimate the volume of polytopes \cite{HL}.  

In this paper, we consider {\em tropical balls} in terms of the {\em tropical metric} over the {\em tropical projective torus} with the max-plus algebra.  A tropical ball $B_l(x_0)$, which is formally defined in Definition \ref{df:tropicalball}, with the radius $l > 0$ at the center $x_0$ is defined as the set of all points over the tropical projective torus with the tropical metric $d_{\rm tr}$ defined in Definition \ref{eq:tropmetric} from $x_0$ less than or equal to $l$.  Put simply, a tropical ball is equivalent to a ball over an Euclidean space by replacing the $l_2$-norm with the tropical metric $d_{\rm tr}$.  A tropical polytope over the tropical projective torus is the smallest {\em tropically convex set} of finitely many points from the tropical projective torus.  This is an analogue of a classical polytope over an Euclidean space, which is the smallest convex set of finitely many points in the Euclidean space.  Similar to a minimum enclosing ball and maximum inscribed ball of a classical polytope, a minimum enclosing tropical ball of a tropical polytope is the smallest tropical ball containing the tropical polytope and a maximum inscribed tropical ball of a tropical polytope is the largest tropical ball contained in the given tropical polytope. 
In this paper we show that the linear programming problem formulated in \eqref{eq:maxball}-\eqref{eq:maxball4} below provides the center, $x_0$, and the radius, $R > 0$, of the maximum inscribed tropical ball and the linear programming problem formulated in~\eqref{eq:minball}-\eqref{eq:minball3} below provides the center, $y_0$, and the radius, $r >0$, of the minimum enclosing tropical ball of a given tropical polytope in the tropical projective torus.

It is hard to compute the volume of a polytope.
In general, a tropical polytope does not have to be classically convex in terms of the Euclidean metric and it can have some components which are of lower dimension than the tropical projective torus \cite{joswigBook} as shown in Figure \ref{fig:apicomplexa}.  
\begin{figure}
    \centering
    \includegraphics[width=0.8\textwidth]{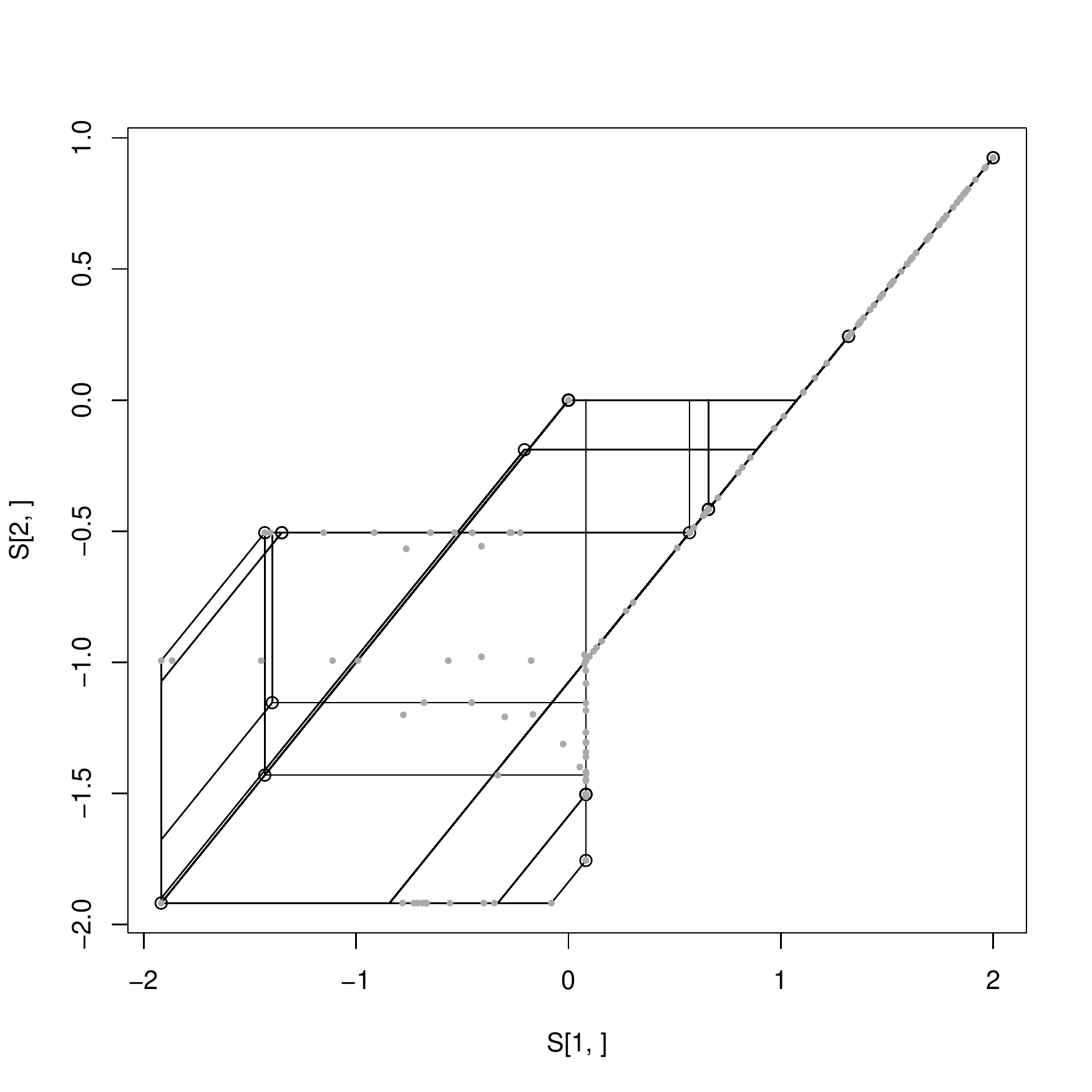}
    \caption{Best-fit tropical triangle over the space of phylogenetic trees for the Apicomplexa genome data from Kuo et al.~\cite{kuo} computed by tropical principal component analysis proposed by Yoshida et al.~\cite{YZZ}. Grey points on the tropical triangle are projected points via the tropical metric from observations in the data.}
    \label{fig:apicomplexa}
\end{figure}
Therefore, computing the Euclidean volume of a tropical polytope over the tropical projective torus is \#-P hard and even estimating the approximate volume of a tropical polytope is NP-hard in general \cite{Gaubert}.  In a Euclidean space, counting the number of lattice points in a classical polytope and computing the volume of the polytope are strongly related by Ehrhart theory \cite{Ehrhart}, and are \#-P hard problems even though an input polytope is described by an oracle \cite{Elekes}.
As a specific case, Brandenburg et al.~\cite{Zhang_Vol} apply tools from toric geometry to compute multivariate versions of volume, Ehrhart and $h^*$-polynomials of lattice {\em polytropes}, which are tropically and classically convex sets of finitely many lattice points.

In Euclidean space, Cousins and Vempala \cite{CV} apply a minimum enclosing ball and Gaussian distributions to estimating the volume of a classical polytope via simulated annealing sampling.  In this paper we develop a novel method to estimate the Euclidean volume of a tropical polytope over the tropical projective torus with the max-plus algebra using the minimum enclosing tropical ball of a tropical polytope, and a Hit-and-Run (HAR) sampler proposed by Yoshida, et al. \cite{Trop_HAR}.  We note that the ratio of the volume of a minimum enclosing tropical ball and the volume of a maximum inscribed tropical ball of a tropical polytope is an approximation of an acceptance rate of a sampler and in general this ratio can be very small.  Therefore, it is still difficult to estimate the volume of a tropical polytope in general.  

Sampling from a tropical polytope itself is an important issue.
Yoshida et al.~\cite{Trop_HAR} propose a HAR sampler using the tropical line segment with the tropical metric to sample from a tropically convex set.  They showed that their HAR sampler with the tropical metric can sample uniformly from a polytrope, so that using a decomposition of a tropical polytope into a set of tropical simplicies, which are polytropes, described in \cite{JoswigKulas+2010+333+352}, one can sample uniformly from a tropical polytope.  However, they have not described explicitly how to sample uniformly from a tropical polytope in general and they do not implement their proposed method explicitly.  Thus, in this paper, we apply a minimum enclosing tropical ball of a tropical polytope combined with a HAR sampler using the tropical metric to sample uniformly from a tropical polytope in general.

 This paper is organized as follows: Section~\ref{BASICS} provides basics of tropical geometry. In section~\ref{sec:balls}, we illustrate how to find the maximum inscribed and minimum enclosing tropical balls for a given tropical polytope.  We remind readers how to construct the hyperplane, or $h^*-representation$, of a tropical polytope with a given vertex set described in \cite{Zhang_Vol}. Using the $h^*-representation$, we show how to compute the maximum inscribed tropical ball using linear programming.  Next, we describe how we can, again, use linear programming to construct the minimum enclosing tropical ball.  In section~\ref{sec:Vol_est}, we show how, by sampling from the minimum enclosing tropical ball using the HAR method described in~\citep{Trop_HAR} we can obtain a volume estimate of the tropical polytope, $P$.  In addition, we show that if $P\in \mathbb{R}^e/\mathbb{R}\mathbf{1}$ is defined by $e$ vertices, methods exist to identify the union of full-dimensional cells, or \textit{(e-1)-trunk} (Definition~\ref{def:trunk}), of a tropical polytope, reducing the radius of the minimum enclosing tropical ball and therefore the number of sample points necessary to obtain a volume estimate of $P$. In section~\ref{sec:uni_samp}, we show that sampling from the minimum enclosing tropical ball and enumerating the tropical simplices defining the defining $P$ of the tropical polytope allows uniform sampling of the $(e-1)$-trunk of any tropical polytope.

 All computational experiments were conducted with {\tt R}, statistical computational tool.  The {\tt R} code used for this paper is available at \url{https://github.com/barnhilldave/Tropical-Balls.git}.

\section{Tropical Basics}\label{BASICS}%Approach

Throughout this paper, we consider the tropical projective torus $\mathbb R^e \!/\mathbb R {\bf 1}$, where\\ ${\bf 1}:= (1, 1, \ldots , 1) \in R^e$ is the vector of all ones. This means that we have
\[
(x_1 + c, x_2 + c, \ldots , x_e+ c) =  (x_1, x_2, \ldots , x_e)
\]
for any $c \in \mathbb{R}$ and for any vector $x:= (x_1, x_2, \ldots, x_e) \in \mathbb R^e \!/\mathbb R {\bf 1}$.
% which is isomorphic to $\R^{e-1}$.
For more details, see \cite{MS} and \cite{joswigBook}.

\begin{definition}[Tropical Arithmetic Operations]
Under the tropical semiring with the {\em max-plus algebra} $(\,\mathbb{R} \cup \{-\infty\},\oplus,\odot)\,$, we have the tropical arithmetic operations of addition and multiplication defined as:
$$x \oplus y := \max\{x, y\}, ~~~~ x \odot y := x + y ~~~~\mbox{  where } x, y \in \mathbb{R}\cup\{-\infty\}.$$
Note that $-\infty$ is the identity element under addition $\oplus$ and $0$ is the identity element under multiplication $\odot$ over this semiring. 

On the other hand, under the tropical semiring with the {\em min-plus algebra} $(\,\mathbb{R} \cup \{\infty\},\boxplus,\odot)\,$, we have the tropical arithmetic operations of addition and multiplication defined as:
$$x \boxplus y := \min\{x, y\}, ~~~~ x \odot y := x + y ~~~~\mbox{  where } x, y \in \mathbb{R}\cup\{\infty\}.$$
\end{definition}

%In some cases, we are interested in the connection between a max-tropical polytope and arithmetic involving min-tropical hyperplanes.  In the cases of min-tropical hyperplanes, we replace $\oplus$ with $\boxplus$ where \[x \boxplus y:= \min\{x,y\}.\]

\begin{definition}[Tropical Scalar Multiplication and Vector Addition]
For any $x,y \in \mathbb{R}\,\cup\, \{-\infty\}$ and for any $v = (v_1, \ldots ,v_e),\; w= (w_1, \ldots , w_e) \in (\mathbb{R}\cup\{-\infty\})^e$, we have tropical scalar multiplication and tropical vector addition defined as:
   % $$a \odot v:= (a + v_1,  \ldots ,a + v_e),$$
$$x \odot v \oplus y \odot w := (\max\{x+v_1,y+w_1\}, \ldots, \max\{x+v_e,y+w_e\}).$$
\end{definition}

\begin{definition}[Tropical Matrix Operations]
    Suppose we have two $(k_1\times k_2)$-dimensional matrices $A,B$ and an $(k_3\times k_1)$-dimensional matrix $C$ with entries in $\mathbb{R}\cup\{-\infty\}$.  Then we can define the tropical matrix operations $A\oplus B$ and $A\otimes C$ in analogy with the ordinary matrix operations. Specifically, we have
\[(A\oplus B)_{ij} = A_{ij}\oplus B_{ij},\ \  (A\otimes C)_{ij} = \bigoplus_{\ell=1}^{k_1} A_{i\ell}\otimes C_{\ell j}.\]
\end{definition}

\begin{definition}[Generalized Hilbert Projective Metric]
\label{eq:tropmetric} 
For any points $v:=(v_1, \ldots , v_e), \, w := (w_1, \ldots , w_e) \in \mathbb R^e \!/\mathbb R {\bf 1}$ where $[e]:=\{1,\ldots,e\}$, the {\em tropical distance} (also known as {\em tropical metric}) $d_{\rm tr}$ between $v$ and $w$ is defined as:
\begin{equation*}
d_{\rm tr}(v,w)  := \max_{i \in [e]} \bigl\{ v_i - w_i \bigr\} - \min_{i \in [e]} \bigl\{ v_i - w_i \bigr\}.
\end{equation*}
%This distance measure is a metric in $\mathbb R^e \!/\mathbb R {\bf 1}$.
\end{definition}

\begin{definition}[Tropical Polytopes]\label{def:polytope}
Suppose we have $S \subset \mathbb R^e \!/\mathbb R {\bf 1}$. If % and suppose 
\[
x \odot v \oplus y \odot w \in S
\]
for any $x, y \in \R$ and for any $v, w \in S$, then $S$ is called {\em tropically convex}.
Suppose $V = \{v^1, \ldots , v^s\}\subset \mathbb R^e \!/\mathbb R {\bf   1}$.  The smallest tropically convex subset containing $V$ is called the {\em tropical convex hull} or {\em tropical polytope} of $V$ which can be written as the set of all tropical linear combinations of $V$
$$ \mathrm{tconv}(V) = \{a_1 \odot v^1 \oplus a_2 \odot v^2 \oplus \cdots \oplus a_s \odot v^s \mid  a_1,\ldots,a_s \in \R \}.$$
The smallest subset $V'\subseteq V$ such that
\[
\mathrm{tconv}(V') = \mathrm{tconv}(V)
\]
is called a minimum set or a generating set with $|V'|$ being the cardinality of $V'$.  For $P=\tconv(V')$ the boundary of $P$ is denoted $\partial P$.  
\end{definition}

We call a {\em max-tropical polytope} if a tropical polytope is defined in terms of the max-plus algebra and we call a {\em min-tropical polytope} if a tropical polytope is defined in terms of the min-plus algebra.

% \begin{definition}
% A {\em tropical simplex} is a tropical polytope that possess a minimum vertex set, $V'$, where $|V'|=e$. A {\em tropical simplex} is denoted $P_\Delta$.
% \end{definition}

% \begin{definition}
%     A {\em tropical line segment} between two points $v^1, \, v^2 \in \mathbb R^e \!/\mathbb R {\bf 1}$ is a tropical polytope, $P$, of a set of two points $\{v^1, \, v^2\} \subset \mathbb R^e \!/\mathbb R {\bf   1}$. The set of bend points along a tropical lines segments is denoted as $B$.
% \end{definition}

\begin{definition}[Polytropes (See~\citep{JoswigKulas+2010+333+352})]
A classically convex tropical polytope is called a {\em polytrope}. 
\end{definition}

\begin{definition}[Covector Decomposition (From~\citep{DS} and~\citep{TVOL_2})]\label{def:cov_D}
    A tropical polytope may be decomposed into a polyhedral complex of polytropes known as a \textit{covector decomposition}  of $P$, denoted as $\mathcal{C}_P$.
\end{definition}

\begin{definition}[Dimension of a Tropical Polytope]
    The dimension of a tropical polytope, $P\in \mathbb{R}^e/\mathbb{R}\mathbf{1}$, is defined by the polytrope of maximal dimension in $\mathcal{C}_P$ and is denoted as $dim(P)$.
\end{definition}

Next we remind the reader of the definition of a projection in terms of the tropical metric onto a tropical polytope.
The tropical projection formula can be found as Formula 5.2.3 in \cite{MS}. %\cite[(5.2.30)]{MS}
\begin{definition}[Projections of Tropical Points]\label{def:proj}
Let $V:= \{v^1, \ldots, v^s\} \subset \mathbb{R}^e/{\mathbb R} {\bf 1}$ and let $P = \tconv(v^1, \ldots, v^s)\subseteq \mathbb R^{e}/\RR{\bf 1}$ be a tropical polytope with its vertex set $V$.
For $x \in \T$, let
\begin{equation}\label{eq:tropproj} %\label{eq:nearestpoint}
\pi_P (x) \!:=\! \bigoplus\limits_{l=1}^s \lambda_l \odot  v^l, ~ ~ {\rm where} ~ ~ \lambda_l \!=\! {\rm min}(x-v^l).
\end{equation}

% \begin{equation}\label{eq:tropproj} %\label{eq:nearestpoint}
% \pi_\mathcal{P} (x) \!:=\! \lambda_1 \odot  v^1 \oplus \lambda_2 \odot  v^2 \oplus \cdots \oplus \lambda_s \odot  v^s, ~ ~ {\rm where} ~ ~ \lambda_l \!=\! {\rm min}(x-v^l).
% \end{equation}
\noindent Then 
\[
d_{\rm tr}(x, \pi_{P} (x))  \leq d_{\rm tr}(x, y)
\]
for all $y \in P$ and $\pi_P(x) \in P$.  In other words, $\pi_P (x)$ is the projection of $x \in \T$ in terms of the tropical metric $d_{\rm tr}$ onto the tropical polytope $P$.
\end{definition}

The following definitions describe the relationship between tropical hyperplanes and tropical polytopes. Specifically, they illustrate the connection between max-tropical polytopes and min-tropical hyperplanes.

% \begin{definition}[Tropical Hyperplane from~\cite{joswigBook}]
% \label{def:trop_hyp} 
% A tropical hyperplane with respect to the max-plus algebra is the tropical hypersurface of a homogeneous linear tropical polynomial
% \begin{equation}
%    H^{\max}_{\omega} := \left\{ x \mid \bigoplus_{i=1}^e \omega_i \odot x_i \right\}
% \end{equation}
% \noindent where $\omega\in\mathbb{R}^e/\mathbb{R}\mathbf{1}$ represents the normal vector of the tropical hyperplane.  
% A tropical hyperplane with respect to the min-plus algebra is the tropical hypersurface of a homogeneous linear tropical polynomial
% \begin{equation}
%    H^{\min}_{\omega} := \left\{ x \mid \op_{i=1}^e \omega_i \odot x_i \right\}.
% \end{equation}
% %The max-tropical hyperplane is denoted $H^{\max}_{\omega}$ and the min-tropical hyperplane is denoted $H^{\min}_{\omega}$ (in which case $u\boxplus v:=\min(u,v)$).

% \end{definition}

\begin{definition}[Tropical Hyperplane from~\cite{YTMM}]\label{def:trop_hyp}
    For any $\omega:=(\omega_1,\ldots,\omega_e)\in\mathbb{R}^e/\mathbb{R}\mathbf{1}$, the max-tropical hyperplane defined by $\omega$, denoted as $H_\omega^{\max}$, is the set of points $x\in \mathbb{R}^e/\mathbb{R}\mathbf{1}$ such that 
\begin{equation}
  % \bigoplus_{i=1}^e \omega_i \odot x_i
  \max_{i \in [e]} \Big\{\omega_i + x_i \Big\}
\end{equation}
\noindent is attained at least twice. Similarly, a min-tropical hyperplane denoted as $H^{\min}_\omega$, is the set of points $x\in \mathbb{R}^e/\mathbb{R}\mathbf{1}$ such that  

\begin{equation}
   %\op_{i=1}^e \omega_i \odot x_i
   \min_{i \in [e]} \Big\{\omega_i + x_i \Big\}
\end{equation}

\noindent is attained twice.  If it is clear from a content, we simply denote $H_\omega$ as a tropical hyperplane in terms of the min-plus or max-plus algebra where $\omega$ is the {\em normal vector} of $H_\omega$.
\end{definition}

\begin{definition}[Sectors from~\cite{YTMM}]\label{def:sector}
    Every tropical hyperplane, $H_{\omega}$, divides the tropical projective torus, $\mathbb{R}^e/\mathbb{R}\mathbf{1}$ into $e$ connected components, which are {\em open sectors}
    \[S^i_\omega := \{x\in \mathbb{R}^e/\mathbb{R}\mathbf{1}\,|\, \omega_i + x_i > \omega_j +x_j, \forall j\neq i\},\; i= 1,\ldots, e.\]
    \noindent These {\em closed sectors} are defined as

    \[\overline{S}^i_\omega := \{x\in \mathbb{R}^e/\mathbb{R}\mathbf{1}\,|\, \omega_i + x_i \geq \omega_j +x_j, \forall j\neq i\},\; i= 1,\ldots, e.\]
\end{definition}
\begin{definition}[Tropical Hyperplane Arrangements]\label{def:arrange}
    For a given set of points, $V=\{v^1,\ldots,v^s\}$,  tropical hyperplanes with apices at each $v^i\in V$ represent the {\em tropical hyperplane arrangement} of $V$, $\mathcal{A}(V)$, where
\[\mathcal{A}(V):=\{H_{-v^1},\ldots,H_{-v^s}\}.\]

\noindent If we consider a collection of tropical hyperplanes defined in terms of the max-plus algebra, we call this arrangement a {\em max-tropical hyperplane arrangement} denoted $\mathcal{A}^{\max}(V)$.  Likewise, considering tropical hyperplanes defined in terms of the min-plus algebra is called a {\em min-tropical hyperplane arrangement} denoted $\mathcal{A}^{\min}(V)$.
\end{definition}

\begin{definition}[Cells]\label{def:cells}
    For a given hyperplane arrangement, $\mathcal{A}(V)$, a {\em cell} is defined as the intersection of a finite number of closed sectors. Cells may be {\em bounded} or {\em unbounded}. Bounded cells are polytropes. 
\end{definition}

\begin{definition}[Bounded Subcomplex (See~\citep{DS})]\label{def:subcomplex}
For a vertex set, $V$,  $\mathcal{A}(V)$ defines a collection of bounded and unbounded cells which is known as a {\em cell decomposition}.  The union of bounded cells defines the {\em bounded subcomplex}, $\mathcal{K}(V)$.   
\end{definition}

\begin{theorem}[Theorem 2.2 in \cite{Dochtermann_2012}]
\label{thm:min_hyp}
A $\max$-tropical polytope, $P$, is the union of cells in $\mathcal{K}(V)$ of the {\em cell decomposition} of the tropical projective torus induced by $\mathcal{A}^{\min}(V)$.
\end{theorem}

Theorem~\ref{thm:min_hyp} describes a $\mathcal{K}(V)$ as a collection of bounded cells induced by $\mathcal{A}(V)$.  Definition~\ref{def:cov_D} describes $P$ in terms of a union of polytropes.  The bounded cells induced by $\mathcal{A}^{\min}(V)$ are these polytropes. Therefore, the union of the polytropes in $\mathcal{C}_P$ defines $\mathcal{K}(V)$. Throughout this paper we are interested in sampling the union of $(e-1)$-dimensional polytropes belonging to $\mathcal{K}(V)$.  The union of $(e-1)$-dimensional polytropes is described in the  following definition.

% \begin{definition}[See~\citep{Gallart2021TropicalBA}]\label{def:trop_simp_comp}
%     Let $V$ be a set of points such that $\tconv(V)=P\in \mathbb{R}^e/\mathbb{R}\mathbf{1}$, the collection of tropical polytopes defined by all subsets $V'\subseteq V$ where $|V'|\leq e$ and a tropical polytope $P'=\tconv(V')$, is known as the tropical simplicial complex, denoted as $\Delta_P$. We say $\Delta_P$ is {\em pure} if all facets in $\Delta_P$ has the same cardinality.
% \end{definition}

% \noindent The previous definitions and theorem show that the boundary of a max-tropical polytope, $P$, is defined by $\mathcal{A}^{\min}(V')$ where $V'\subseteq V$ is the minimal vertex such that $\tconv(V')=\tconv(V)$.   The intersection of closed sectors of the associated$H^{\min}_{-v^i}$ for all $v^i \in V$ defines the bounded subcomplex $\mathcal{K}(V)$ which results in a collection of cells, or polytropes, known as a polyhedral complex. This polyhedral complex is what we call the \textit{covector decomposition} of $P$, or $\mathcal{C}_P$~\citep{FINK2015291}. 

\begin{definition}[{\em i-trunk} and {\em i-tentacles} (Definition 2.1 in~\citep{TVOL_2})]\label{def:trunk}
    Let $P$ be a tropical polytope and let $i\in[e-1]$ where $[e-1]=\{1,\ldots,e-1\}$.  Let $\mathcal{F}_P$ be the family of relatively open tropical polytopes in $\mathcal{C}_P$. For any $T\in \mathcal{F}_P$, $T$ is called an {\em{i-tentacle element}} of $\mathcal{F}_P$ if it is not contained in the closure of any $(i+1)$-dimensional tropical polytope in $\mathcal{F}_P$ where the dimension of $T$ less than or equal to $i$.  The {\em{i-trunk}} of $P$, is defined as
    \[Tr_i(P):=\bigcup\Big\{F\in\mathcal{F}_P: \exists \;G\in\mathcal{F}_P \text{ with } \dim(G)\geq i \text{ such that } F\subseteq G\Big\}\]
where $dim(G)$ is the dimension of $G \subset \mathcal{F}_P$.
    \noindent The $Tr_i(P)$ represents the portion of the $\mathcal{K}(V)$ with $(i-1)$-tentacles removed.  The minimum enclosing ball containing containing only $Tr_i(P)\subseteq P$ is denoted $B_k(Tr_i(P))$.
\end{definition}

\begin{remark}[Proposition 2.3 in~\citep{Zhang_Vol}]
    The $Tr_{e-1}(P)$ of a tropical polytope, $P\in\mathbb{R}^e/\mathbb{R}\mathbf{1}$ is a tropical polytope.
\end{remark}

\begin{example}
Consider the tropical polytope, $P=\{(0,0,0),(0,-1,1),(0,2,2),(0,1,-1)\}$.  The $Tr_2(P)$ is the gray portion shown in Figure~\ref{fig:trunk_ex}.
% [Example of $i$-trunk.  RY]
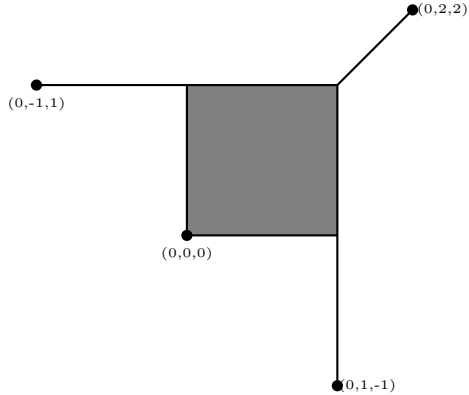
\begin{figure}[H]
    \centering

\begin{tikzpicture}[node distance={30mm}, thick, main/.style = {draw, circle}]
\fill[gray] (1,1) rectangle (3,3);
\node[] (2) at (-1,2.75) {\tiny(0,-1,1)};
\node[] (5) at (1,.75) {\tiny(0,0,0)}; 
\node[] (6) at (3.4,-1) {\tiny(0,1,-1)};
\node[] (7) at (4.4,4) {\tiny(0,2,2)};
\draw [-](1,1) to (1,3);
\draw [-](1,1) to (3,1);
\draw [-](1,3) to (3,3);
\draw [-](3,1) to (3,3);
\draw [-](1,3) to (-1,3);
\draw [-](3,1) to (3,-1);
\draw [-](3,3) to (4,4);
\node at (1,1)[circle,fill,inner sep=1.5pt]{};
\node at (-1,3)[circle,fill,inner sep=1.5pt]{};
\node at (3,-1)[circle,fill,inner sep=1.5pt]{};
\node at (4,4)[circle,fill,inner sep=1.5pt]{};
\end{tikzpicture}
    \caption{A tropical polytope, $P$, in $\mathbb{R}^3/\mathbb{R}\mathbf{1}$ defined by four vertices.  The $Tr_{2}(P)$ is the portion in gray.} 
    \label{fig:trunk_ex}
\end{figure}
\end{example}
All methods described in this paper involve sampling from the $Tr_{e-1}(P)$ of a tropical polytope $P$.  However, tropical polytopes can exist without a $Tr_{e-1}(P)$ as shown by Joswig et al. (See Figure 1 (left) in~\citep{JoswigKulas+2010+333+352}). Therefore, all tropical polytopes considered in this research are assumed to contain a $Tr_{e-1}(P)$.

\section{Using Optimization to Find Maximum Inscribed and Minimum Enclosing Tropical Balls}\label{sec:balls}

% Let $B_R(P)$ be the maximum inscribed tropical ball with its radius $R > 0$ and $B_r(P)$ be the minimum enclosing tropical ball with its radius $r > 0$. 

In this section we illustrate how to find the hyperplane-representation of a {\em tropical simplex} $P_\Delta$, in order to compute the {\em maximum incscribed ball} of $P_\Delta$. We then show how to calculate the {\em minimum enclosing ball} of any tropical polytope $P$, using a linear programming formulation.  Tropical simplices and tropical balls will be leveraged throughout this paper thus we offer the following definitions.

\begin{definition}[Tropical Simplex]
    A tropical simplex is a tropical polytope that possess a minimum vertex, or generating, set $V'$ such that $|V'| = e$. A tropical simplex is denoted $P_\Delta$.
\end{definition}

\begin{definition}[Tropical Ball]\label{df:tropicalball}
    A tropical ball, $B_l(x_0)$, around $x_0 \in \mathbb{R}^e/\mathbb{R}{\bf 1}$ with a radius $l > 0$ is defined as follows:
\[
B_l(x_0) = \{y \in \mathbb{R}^e/\mathbb{R}{\bf 1}\, |\, d_{\rm tr}(x_0, y) \leq l\}.
\]
For a tropical polytope, $P$, the {\em minimum enclosing ball}, denoted as $B_r(P)$, is the tropical ball of smallest radius $r$ that fully contains $P$.  For the same $P$, the tropical ball with maximum radius $R$, that is fully contained in $P$ is called the {\em maximum inscribed ball} and is denoted $B_R(P)$.  
\end{definition}

\subsection{Hyperplane Representation of Tropical Polytopes with the Max-Plus Algebra}

Brandenburg et al.~in~\cite{Zhang_Vol} show that the hyperplane-representation of a polytrope $P$, denoted as $h^*-representation(P)$, can be constructed from an associated \textit{Kleene Star} weight matrix, $\mathbf{m^*}$. This is shown in the following definition and extends it to any $P_\Delta$.

\begin{definition}[Hyperplane Representation of a Tropical Simplex (See Proposition 2.14 in~\citep{Zhang_Vol})]\label{def:h_rep}
    For a tropical simplex that is a polytrope, $P_\Delta\in \mathbb{R}^e/\mathbb{R}\mathbf{1}$, $P_\Delta$ may be defined by the intersection of a collection of classical half-spaces.  These half-spaces are constructed using a \textit{Kleene Star} weight matrix, $\mathbf{m^*}$, where each $m_{ij}$ is the $(i,j)$-th entry in $\mathbf{m^*}$.  The half-spaces defining $P_\Delta$ are constructed as follows
    \begin{equation}\label{eq:maxhyperp}
    P_\Delta=\{y\in\mathbb{R}^e\;|\; y_j-y_i\leq -m_{ij}, y_1=0, m_{ij}\in \mathbf{m^*}, i \neq j\}.
\end{equation}

\noindent We denote this intersection of half-spaces as $h^*-representation(P_\Delta)$.  By contrast, the vertex representation of $P_\Delta$ using a vertex set, $V$, is denoted as $v-representation(P_\Delta)$.  For a tropical simplex, $P_\Delta$, that is not classically convex, the $h^*-representation(P_\Delta)$ only defines the {\em $(e-1)$-trunk} of $P$ (See Definition~\ref{def:trunk}). 
\end{definition}

For any weight matrix $\mathbf{m}\in\mathbb{R}^{e\times e}$ associated with a complete graph of $e$ nodes with no positive cycles, $m_{ij}$ represents the weight from node $i$ to $j$ with the diagonal being all zeroes. The weight matrix, $\mathbf{m^*}$, can be constructed by taking the $(e-1)$-th tropical power of $\mathbf{m}$. that is, $\textbf{m}^{\odot e-1}$ with each $m_{ij}\in \mathbf{m^*}$ being the maximal path from node $i$ to node $j$~\citep{LP,TranKleene}.  From the new matrix $\mathbf{m^*}$, we can obtain $h^*-representation(P)$ using~\eqref{eq:maxhyperp}~\citep{Zhang_Vol}.

 In this paper we start with a vertex set, $V'$, defining a tropical simplex, $P_\Delta=\tconv(V')$ and where we assume that  $V'$ is a minimum vertex set.  Using $V'$ we then construct $\mathbf{m^*}$.  By Proposition 2.14 in~\citep{Zhang_Vol}, there is a correspondence between the $\mathbf{m^*}$ and the points in $V'$.  
 
 To build $\mathbf{m^*}$, we must first calculate the {\em tropical determinant} of a matrix $A$, using Definition~\ref{def:trop_det}, where column vectors, $A_j$, are tropical points $V'$. 

\begin{definition}[Tropical Determinant]\label{def:trop_det}
    Let $q$ be a positive integer.  For any squared tropical matrix $A$ of size $q \times q$ with entries in $\mathbb{R}\;\cup\;\{-\infty\}$, the tropical determinant of $A$ is defined as:

    \[tdet(A):=\max_{\sigma\in S_q}\{A_{\sigma(1),1}+A_{\sigma(2),2}+\ldots+A_{\sigma(q),q}\},\]
where $S_q$ is all the permutations of $[q]:=\{1,\ldots,q\}$, and $A_{i,j}$ denotes the $(i,j)$-th entry of A.  The tropical matrix is singular if:
\begin{enumerate}
    \item A is non-square;
    \item the $tdet(A)=-\infty$; or
    \item at least two permutations achieve the maximum $tdet(A)$.
\end{enumerate}
For a square matrix, it is equivalent to saying the row or column vectors lie in a tropical hyperplane (see Proposition 3.4 in~\citep{joswigBook}).
\end{definition}

Finding $tdet(A)$ can be reduced to evaluating a \textit{linear assignment} problem where $e$ tasks are assigned to $e$ workers to achieve maximum benefit. Each $i,j \in [e]$, the $(i, j)$th element of the matrix $A$, $A_{i,j}$, represents the benefit gained from assigning task $i$ to worker $j$.  The $tdet(A)$ provides the permutation that achieves this maximum benefit~(See Observation 3.1 in~\citep{joswigBook}).

After calculating $tdet(A)$, let $\sigma_P$ be the permutation %we reorder the columns of $A$ based on the permutation $\sigma$ 
satisfying $tdet(A)$ as shown in Definition~\ref{def:trop_det}.  
The output of Algorithm \ref{alg:tdet} is an $e \times e$ tropical matrix $A'$ whose diagonal is $A_{\sigma_P(i),i}$.
%This reordering results in each $A_{\sigma_P(i),i}$ sitting on the diagonal of the new matrix, $A'$. We illustrate this method in Algorithm~\ref{alg:tdet}.

\begin{algorithm}[H]
\caption{Calculating the $tdet(A)$ and $\sigma_P\in S_e$ for a square matrix, $A$.} \label{alg:tdet}
\begin{algorithmic}
\State {\bf Input:} A square matrix, $A$, with the column vectors representing the set of vertices in $V'$.
\State {\bf Output:} $tdet(A)$ and $\sigma_P \in S_e$.
\State Enumerate each permuation, $\sigma_k$, of the row indices of $A$ where $\sigma_k \in S_e$.
\State Calculate $t(\sigma_k)=\sum\limits_{i=1}^e A_{\sigma_k(i),i}$ for each $\sigma_k\in S_e$.
\State Let $tdet(A)=$\text{$\max\limits_{k}\,$}$(t(\sigma_k))$ and $\sigma_P=\sigma_k$\\
\Return $tdet(A)$ and $\sigma_P$.
\end{algorithmic}
\end{algorithm}

\begin{example}\label{ex:tdet}
    Consider the polytrope, $P_\Delta:=\tconv(\{(0,0,0),(0,2,5), (0,3,1)\})$.  Letting the vertices be the column vectors of a square matrix $A$, we find $tdet(A)=8$ with $\sigma_P=(1,3,2)$ where the index of $\sigma_P$ represents the column and the value at that index represents the row index of that column. This tells us that coordinates of the matrix, $A$, contributing to $tdet(A)$ are $A_{1,1}=0$, $A_{3,2}=5$, and $A_{2,3}=3$.  Reordering the columns from smallest row index to largest yields the new matrix, $A'$, as follows,
\[A=\left(\begin{array}{ccc}
    0 & 0 & 0\\
    0 & 2 & 3\\
    0 & 5 & 1
\end{array}\right)\rightarrow 
\left(\begin{array}{ccc}
    0 & 0 & 0\\
    0 & 3 & 2\\
    0 & 1 & 5
\end{array}\right)=A'.
\]
\end{example}

The final step to find $\mathbf{m^*}$ is to add $-A'_{kk}\in A'$ to each element of the column, $A'_k$ for $k \in [e]$.  The result of this operations finalizes the construction of $\mathbf{m^*}$.  Once $\mathbf{m^*}$ is constructed, we can extract $h^*-representation(P_\Delta)$ using~\eqref{eq:maxhyperp}. To illustrate extracting $h^*-representation(P_\Delta)$, we use Algorithm~\ref{alg:hyperplane_rep}.

% \begin{equation}\label{eq:maxhyperp}
%     P=\{y\in\mathbb{R}^e\;|\; y_j-y_i\leq -m_{ij}, y_1=0, m_{ij}\in \textbf{m}, i \neq j\}.
% \end{equation}

% \noindent Maintaining the same direction of the inequality as shown in Equation~\eqref{eq:hyperp} requires that we divide the left and right side by $-1$ resulting in Equation~\eqref{eq:maxhyperp}.   

\begin{algorithm}[H]
\caption{Extracting $h^*-representation(P_\Delta)$ of a tropical simplex $P\in\mathbb{R}^e/\mathbb{R}\mathbf{1}$.} \label{alg:hyperplane_rep}
\begin{algorithmic}
\State {\bf Input:} A square matrix, $A$, representing a set of vertices in minimum vertex set $V'$ where the vertices are the columns of $A$.
\State {\bf Output:} $h^*-representation(P_\Delta)$.
\State Calculate $tdet(A)$ and identify $\sigma_P \in S_e$ defining $tdet(A)$ using Algorithm~\ref{alg:tdet}.
\State Reorder the columns to form a new matrix $A'$ such that $A_{\sigma_P(i),i}=A'_{kk}$ where $k=\sigma_P(i)$.
\State For a matrix, $\mathbf{m^*}\in \mathbb{R}^{e \times e}$, let $m_{jk}=A'_{jk}-A'_{kk}$ for all $j, k\in [e]$.
\State Compute $h^*-representation(P_\Delta)$ from, $\mathbf{m^*}$, using~\eqref{eq:maxhyperp}.\\
\Return $h^*-representation(P_\Delta)$.
\end{algorithmic}
\end{algorithm}

% We illustrate how to employ Algorithm~\ref{alg:hyperplane_rep} with the following example.

\begin{example}[Example~\ref{ex:tdet} cont'd]\label{ex:poly_hyper}
Continuing with the polytrope, $P_\Delta$, from the previous example we move to the next step to construct the $h^*-representation(P_\Delta)$. After finding the $tdet(A)$ and reordering the columns to form the new matrix, $A'$, $-A'_{kk}$ is added to each element of the column vector $A'_k$ to construct $\mathbf{m^*}$.  The result follows,

\[\left(\begin{array}{ccc}
    0 & -3 & -5\\
    0 &  0 & -3\\
    0 & -2 &  0
\end{array}\right)=\mathbf{m^*}.\]
  The $h^*-representation(P_\Delta)$, as described in~\eqref{eq:maxhyperp}, is shown below

\begin{eqnarray*}
    y_2 - y_1 \leq& 3\\
    y_3 - y_1 \leq& 5\\
    y_1 - y_2 \leq& 0\\
    y_3 - y_2 \leq& 3\\
    y_1 - y_3 \leq& 0\\
    y_2 - y_3 \leq& 2\\
    y_1 =&\;0.
\end{eqnarray*}
\end{example}
\noindent Figure~\ref{fig:max_kleene} shows the directed graph and polytrope defined by $\mathbf{m^*}$.  The directed graph represents the maximum distance between any two nodes as determined by $\mathbf{m^*}$.

While \cite{Zhang_Vol}
specifically addresses $h^*-representation(P_\Delta)$ for polytropes, this method can also be used for any tropical simplex.  However, the $h^*-representation(P_\Delta)$ only defines the $Tr_{e-1}(P_\Delta)$, resulting in several redundant constraints.
    \begin{figure}[H]
\centering
\begin{tikzpicture}[node distance={30mm}, thick, main/.style = {draw, circle}]
\node[main] (1) at (1,1) {\scriptsize$1$}; 
\node[main] (3) at (3,4) {\scriptsize$3$}; 
\node[main] (2) at (5,1) {\scriptsize$2$};
\draw [->](2) to [out=200,in=340,looseness=1] node [anchor=north] {\scriptsize$0$} (1);
\draw [->](1) to [out=80,in=210,looseness=1] node [anchor=south east] {\scriptsize$-5$} (3);
\draw [->](3) to [out=330,in=100,looseness=1] node [anchor=south west] {\scriptsize$-2$} (2);
\draw [->](3) to [out=250,in=40,looseness=1] node [anchor=north west] {\scriptsize $0$}(1);
\draw [->](1) to [out=20,in=160,looseness=1] node [anchor=south] {\scriptsize$-3$}(2);
\draw [->](2) to [out=140,in=290,looseness=1] node [anchor=north east] {\scriptsize$-3$}(3);
\end{tikzpicture}~
\begin{tikzpicture}[node distance={30mm}, thick, main/.style = {draw, circle}]
\fill[gray] (1,1) rectangle (3,3);
\fill[gray] (1,2) rectangle (2,2.5);
\fill[gray] (4,5) rectangle (3.5,3);
\fill[gray] (3,3) rectangle (4,5);
\fill[gray] (4,2) rectangle (2.9,3.1);
\draw[gray, fill] (1,3) -- (3,5) -- (3,3) -- cycle;
\draw[gray, fill] (3,1) -- (4,2) -- (3,2) -- cycle;
\node[] (2) at (0.75,.75) {\tiny(0,0,0)};
\node[] (5) at (2.75,5.15) {\tiny(0,2,5)}; 
\node[] (6) at (4.4,2) {\tiny(0,3,1)};
\draw [-](1,1) to (1,3);
\draw [-](1,1) to (3,1);
\draw [-](1,3) to (3,5);
\draw [-](3,5) to (4,5);
\draw [-](4,5) to (4,2);
\draw [-](4,2) to (3,1);
\node at (1,1)[circle,fill,inner sep=1.5pt]{};
\node at (4,2)[circle,fill,inner sep=1.5pt]{};
\node at (3,5)[circle,fill,inner sep=1.5pt]{};
\end{tikzpicture}
\caption{Directed graph determined by the \textit{Kleene Star} for Example~\ref{ex:poly_hyper}.}\label{fig:max_kleene}
\end{figure}
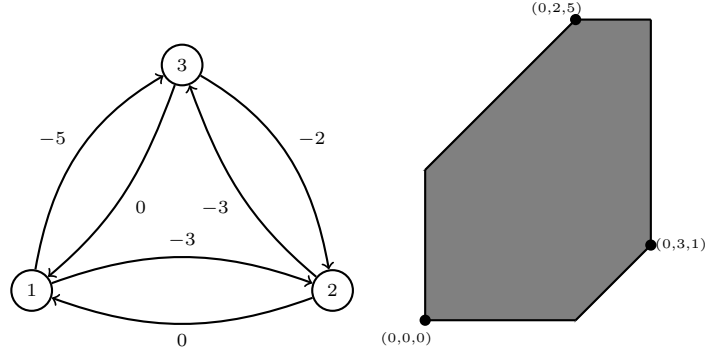

\subsubsection{Maximum Inscribed Balls for Tropical Polytopes}\label{sec:max_balls}

We now illustrate how to use~\eqref{eq:maxhyperp} to formulate a linear program to construct the maximum inscribed ball for a tropical simplex $P_\Delta$ denoted as $B_R(P_\Delta)$.  Like any polytrope, a tropical ball, $B_l(x_0)\in \mathbb{R}^e/\mathbb{R}\mathbf{1}$ can be defined as a tropical simplex with a minimum vertex set $V'$ such that $B_l(x_0)=\tconv(V')$.  

Given a point, $x_0:=(x_0^1,\ldots, x_0^e)$, representing the center of a $B_l(x_0)$, we can obtain vertices in $V'$ for $B_l(x_0)$ such that $B_l(x_0)=\tconv(V')$ where 

% $B_l(x_0)=\{(x_0^1-l,x_0^2-l,x_0^3-l\ldots,x_0^e-l),(x_0^1+l,x_0^2,x_0^3,\ldots,x_0^e),\ldots,(x_0^1,x_0^2,x_0^3,\ldots,x_0^e+l)\}$.   

\[V'=\left\{\begin{array}{ccc}
    (x_0^1-l,x_0^2-l,x_0^3-l\ldots,x_0^e-l),\\(x_0^1+l,x_0^2,x_0^3,\ldots,x_0^e),\\
    \vdots\\
    (x_0^1,x_0^2,x_0^3,\ldots,x_0^e+l)
\end{array}\right\}.\]

\begin{example}
Let $x_0\in \mathbb{R}^3/\mathbb{R}\mathbf{1}$ represent the center of a tropical ball, $B_l(x_0)$ with radius, $l$ where $x_0:=(0,x_0^1,x_0^2)$.  Because $B_l(x_0)$ is a polytrope, the minimum vertex set, $V'$, consists of three vertices, where $B_l(x_0)=\tconv(V')$.  Each $v^i\in V'$ where $i\in[3]$, can be represented in terms of $x_0$ and $l$.  Specifically,
\begin{align*}
    v^1 &= (0, x_0^1-l,x_0^2-l)\\
    v^2 &= (0, x_0^1,x_0^2+l)\\
    v^3 &= (0, x_0^1+l,x_0^2)
\end{align*}

\noindent Figure~\ref{fig:Ball_ag} illustrates the coordinate construction of  $B_l(x_0) \in \mathbb{R}^3/\mathbb{R}\mathbf{1}$.  

\begin{figure}[H]
\centering
\begin{tikzpicture}[node distance={30mm}, thick, main/.style = {draw, circle}]
\draw [-](1,1) to (1,3);
\draw [-](1,1) to (3,1);
\draw [-](1,3) to (3,5);
\draw [-](3,1) to (5,3);
\draw [-](5,5) to (3,5);
\draw [-](5,5) to (5,3);
\node at (1,1) [circle,fill,gray,inner sep=1.5pt]{};
\node at (5,3)[circle,fill,gray,inner sep=1.5pt]{};
\node at (3,5)[circle,fill,gray,inner sep=1.5pt]{};
\node at (3,3)[circle,fill,black,inner sep=1.5pt]{};
\node[] at (.25,1) {\tiny$\left(0,x_0^1-l,x_0^2-l\right)$};
\node[] at (2.35,5) {\tiny$\left(0,x_0^1,x_0^2+l\right)$}; 
\node[]  at (5.6,3) {\tiny$\left(0,x_0^1+l,x_0^2\right)$};
\node[]  at (3,3.3) {\tiny$\left(0,x_0^1,x_0^2\right)$};
\end{tikzpicture}
\caption{Tropical ball, $B_l(x_0)\in \mathbb{R}^3/\mathbb{R}\mathbf{1}$ with radius $l$.}
\label{fig:Ball_ag}
\end{figure}
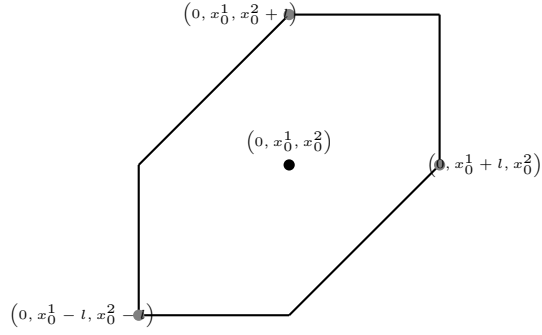
\end{example}

Using the vertices in $V'$, we can compute $B_R(P_\Delta)$. To determine if $B_l(x_0)$ falls inside $P_\Delta$, it suffices to check if the vertices in $V'$ defining $B_l(x_0)$ are in the $P_\Delta$.  Therefore, to find $B_R(P_\Delta)$, we check that each vertex defining $B_R(P_\Delta)$ is inside of the tropical polytope $P_\Delta$ with the center point, $x_0$ %of $B_R(P_\Delta)$
and the maximum radius, $R$. The linear program shown in~\eqref{eq:maxball}-\eqref{eq:maxball4} finds $B_R(P_\Delta)$ by constructing the $h^*-representation(B_R(P_\Delta))$ in terms of its center, $x_0$, as shown in Figure~\ref{fig:Ball_ag}.

\begin{align}\label{eq:maxball}
    &\text{$\max_{R,x_0}$} \quad
    R \\\label{eq:maxball1}
    &\;\text{s.t.} \quad
\;(x_0^i+R)-x_1\le -m_{1i}& \quad
    \forall\;i\in\{2,\dots,e\}\\\label{eq:maxball2}
    &\;x_1-(x_0^j-R)\le -m_{j1}& \quad
    \forall\;j\in\{2,\ldots, e\}\\\label{eq:maxball3}
    &\;(x_0^j+R)-x_0^i \leq -m_{ij}&\quad
    \forall\;i,j\in\{2,\ldots, e\},\;i\neq j\\\label{eq:maxball4}
    &\;x_0^1=0.
\end{align}

\noindent To check if a vertex is inside $P_\Delta$, we leverage (\ref{eq:maxhyperp}).
Because the first coordinate is special (i.e., $x_1=0$), we separately consider the constraint in (\ref{eq:maxhyperp}) for $j=1$, $i=1$ or the other cases.
Note that inequalities are mostly redundant for the vertices in $V'$ defining $B_R(P_\Delta)$.
Therefore only the non-redundant inequality is kept finally in the formulation above.

Notably, the $B_R(P_\Delta)$ need not to be unique.  Results from the following example show that $B_R(P_\Delta)$ can have a center point, $x_0$, from $(0,1.5,1.5)$ to $(0,1.5,3)$ and achieve the desired result.

\begin{example}\label{ex:min_insc1}
    We again consider the tropical simplex, $P_\Delta$, defined in Example~\ref{ex:poly_hyper}.  The formulation to find $B_R(P_\Delta)$ is

    \begin{align*}
    &\text{$\max_{R,x_0}$} \quad
    R \\
    &\;\text{s.t.} \quad
\;x_0^1-(x_0^2-R)\le 0 \\
&\;x_0^1-(x_0^3-R)\le 0\\
    &\;(x_0^2+R)-x_0^1\le 3\\ % 2
    &\;(x_0^3+R)-x_0^1\le 5\\ % 1
    &\;(x_0^3+R)-x_2 \leq 3\\ % -2
    &\;(x_0^2+R)-x_0^3 \leq 2\\ % -3
    &\;x_0^1=0.
\end{align*}

\noindent This results in $x_0=(0,1.5,1.5)$ and radius, $R=1.5$.    Figure~\ref{fig:Ball_ins1} shows $B_R(P)$.
\end{example}

\begin{figure}[H]
    \centering
    \includegraphics[width=0.65\textwidth]{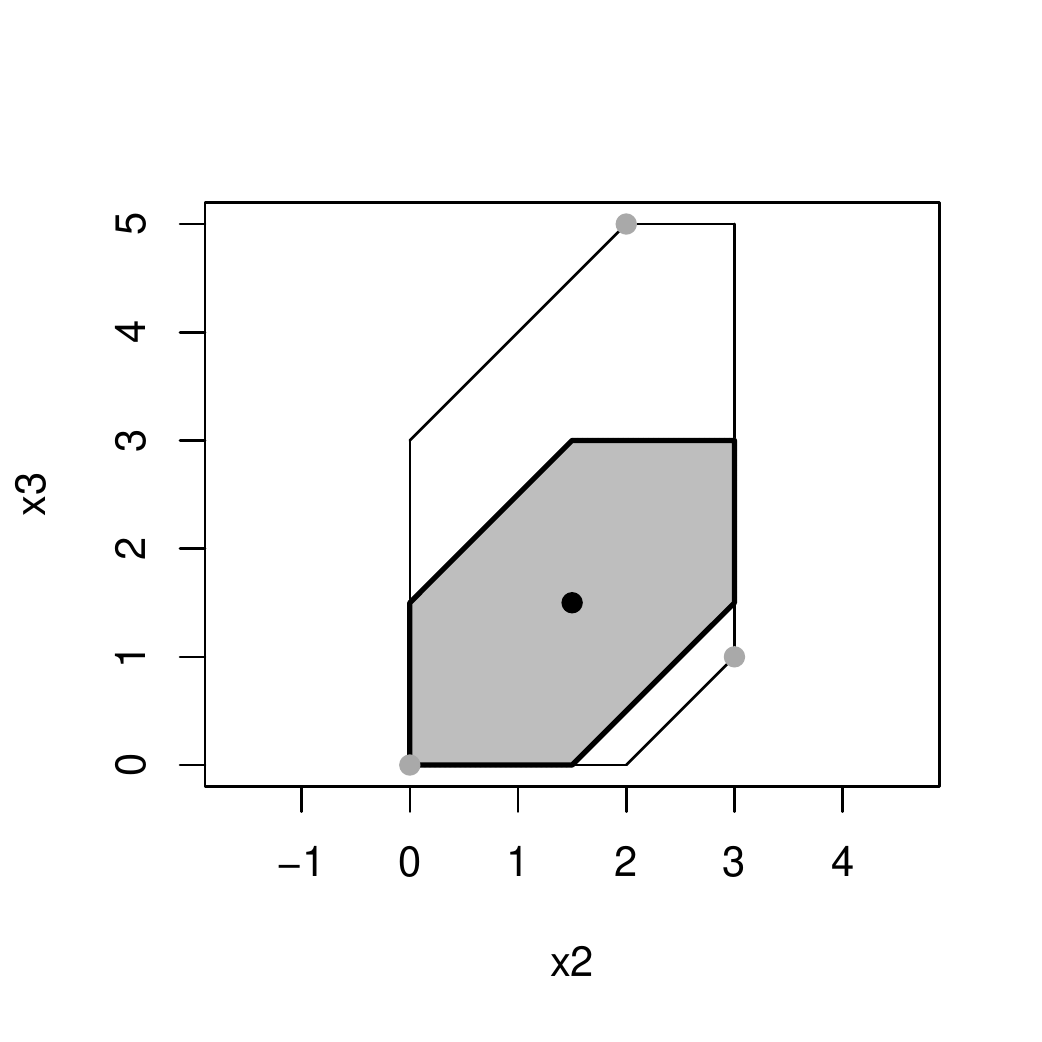}
    \caption{Maximum inscribed ball from Example~\ref{ex:min_insc1}.  The tropical ball has a radius, $R=1.5$ with $B_R(P_\Delta)=\{(0,0,0),(0,1.5,3),(0,3,1.5)\}$ and $x_0=(0,1.5,1.5)$ shown in black.  The points defining $P$ are shown in gray.  Note that the lower left point for $P$ and $B_R(P_\Delta)$ are coincident.}
    \label{fig:Ball_ins1}
\end{figure}

\subsubsection{Maximum Inscribed Balls for General Tropical Polytopes}

In the previous section, we focus on computing maximum inscribed balls for tropical simplices.  In this section, we show that we can apply the same process and formulation shown in~\eqref{eq:maxball}-~\eqref{eq:maxball4} to any tropical polytopes $P$ to find $B_R(P)$. To do this we focus on the set of tropical simplices in the {\em tropical simplicial complex} of $P$, $\Delta_P$,  determined by the $t=\binom{n}{e}$ combinations of vertices in the minimum vertex set of $P$, $V'$ with $|V'|=n$~\citep{Gallart2021TropicalBA}. 

\begin{definition}[Tropical Simplicial Complex (See~\citep{Gallart2021TropicalBA})]\label{def:trop_simp_comp}
    Let $V$ be a set of points such that $\tconv(V)=P\in \mathbb{R}^e/\mathbb{R}\mathbf{1}$, the collection of tropical polytopes defined by all subsets $V'\subseteq V$ where $|V'|\leq e$ and a tropical polytope $P'=\tconv(V')$, is known as the tropical simplicial complex, denoted as $\Delta_P$. We say $\Delta_P$ is {\em pure} if all facets in $\Delta_P$ has the same cardinality.
\end{definition}

Only tropical polytopes of $e$ or more vertices may contain a $Tr_{(e-1)}(P)$.  Therefore we apply~\eqref{eq:maxball}-\eqref{eq:maxball4} to each tropical simplex in $\Delta_P$, denoted $P_\Delta^i \in \Delta_P$, to construct $B_{R_i}(P_\Delta^i)$ where $R_i$ is the radius of the maximum inscribed ball for $P_\Delta^i$.  Since $\bigcup_{i=1}^{t}P_\Delta^i=P$, we have $\bigcup_{i=1}^{t}Tr_{e-1}(P_\Delta^i)=Tr_{e-1}(P)$.  Therefore,
$B_R(P)=B_{R_i}(P_i)$ for the $P_i \in \Delta_P$ where $R_i$ is the largest.

% only tropical simplices can contain a $Tr_{e-1}(P)$.  The tropical simplices are  . 

% This $\Delta_P$, is \textit{pure} so all tropical polytopes of fewer than $e$ vertices are faces of at least one tropical simplex, $P_i' \in \Delta_P$. 

% Since   After using the linear program from \eqref{eq:maxball}-\eqref{eq:maxball4} to find $B_R(P'_i)$ for each $P'_i$, $B_R(P)$ will be the $B_R(P'_i)$ with the largest radius, $R$. 

% Definition~\ref{def:cov_D} shows that the covector decomposition of $P$, $\mathcal{C}_P$,  represents the set of all polytropes such that $\bigcup_{i=1}^s P_i=P$ where $P_i \in \mathcal{C}_P$ represents one of these polytropes. 
% In decomposing a tropical polytope in this way, we could calculate the $B_R(P)$ for each $(e-1)$-dimensional polytrope, $P_i \subset \mathcal{C}_P$.  The $B_R(P)$ for $P$ would be the one found for one of the polytropes defined in $\mathcal{C}_P$`\cite{BEZ}.

% Unfortunately, evaluating each $(e-1)$-dimensional polytrope requires enumerating {\em pseudo-vertices} in $P$ (See Definition~\ref{def:pseudovert}).  This is a task that is difficult regardless of dimension. 

\begin{example}[Borrowed from~\citep{TVOL_2}]\label{ex:nc_tpoly}
    Consider a tropical polytope $P:=\tconv(\{(0,-2,5),(0,-2,3),$\\$(0,2,2),(0,1,0)\})$.  Figure~\ref{fig:tpolex} illustrates this tropical polytope.

\begin{figure}[H]
    \centering
    \includegraphics[width=0.65\textwidth]{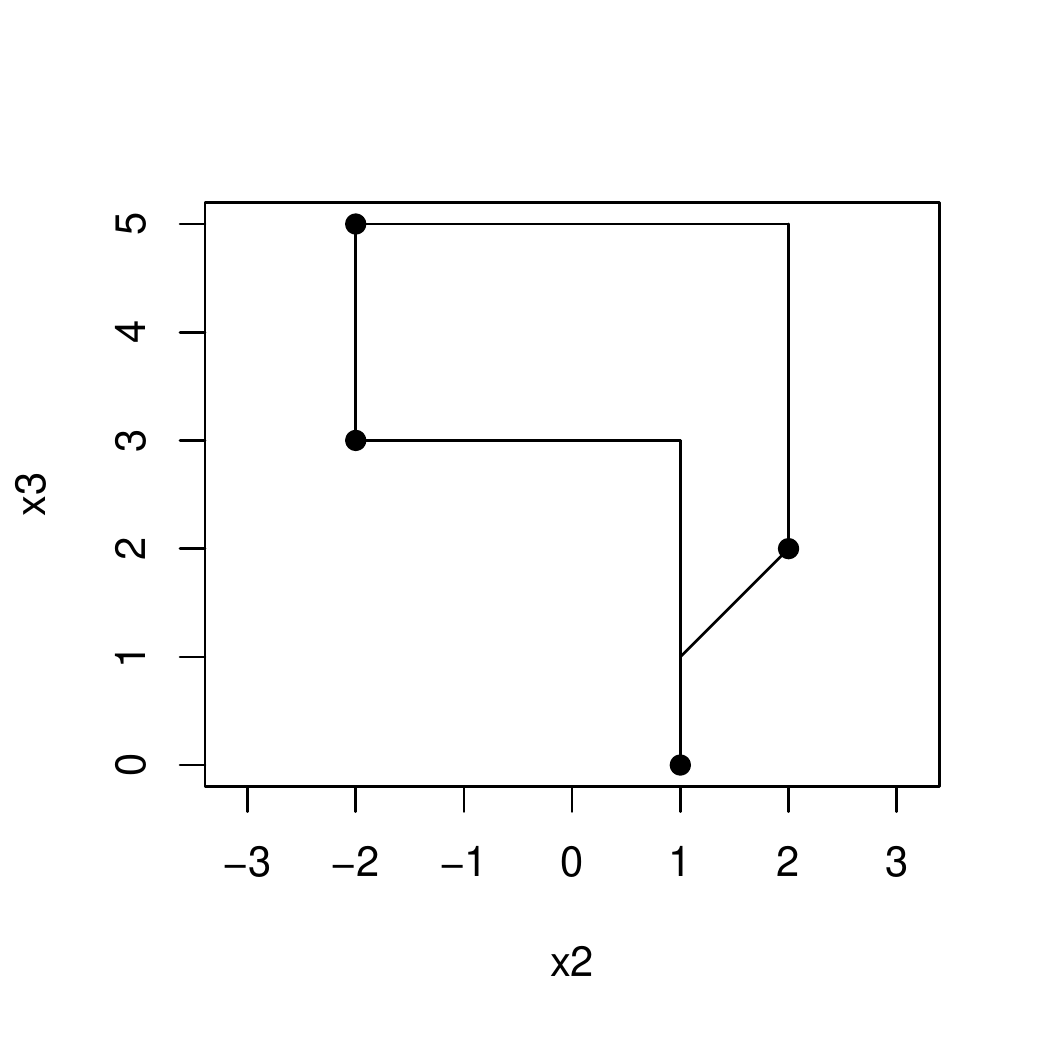}
    \caption{Tropical polytope used in Example~\ref{ex:nc_tpoly}.  The tropical convex hull is defined by the four vertices in black~\citep{TVOL_2}.}
    \label{fig:tpolex}
\end{figure}

\noindent There are four tropical simplices in $\Delta_P$,  so to compute $B_R(P)$ we must apply~\eqref{eq:maxball}-\eqref{eq:maxball4} to evaluate each one.  Figure~\ref{fig:tpolex1} shows the four tropical simplices.

\begin{figure}[H]
 \centering
 \includegraphics[width=0.45\textwidth]{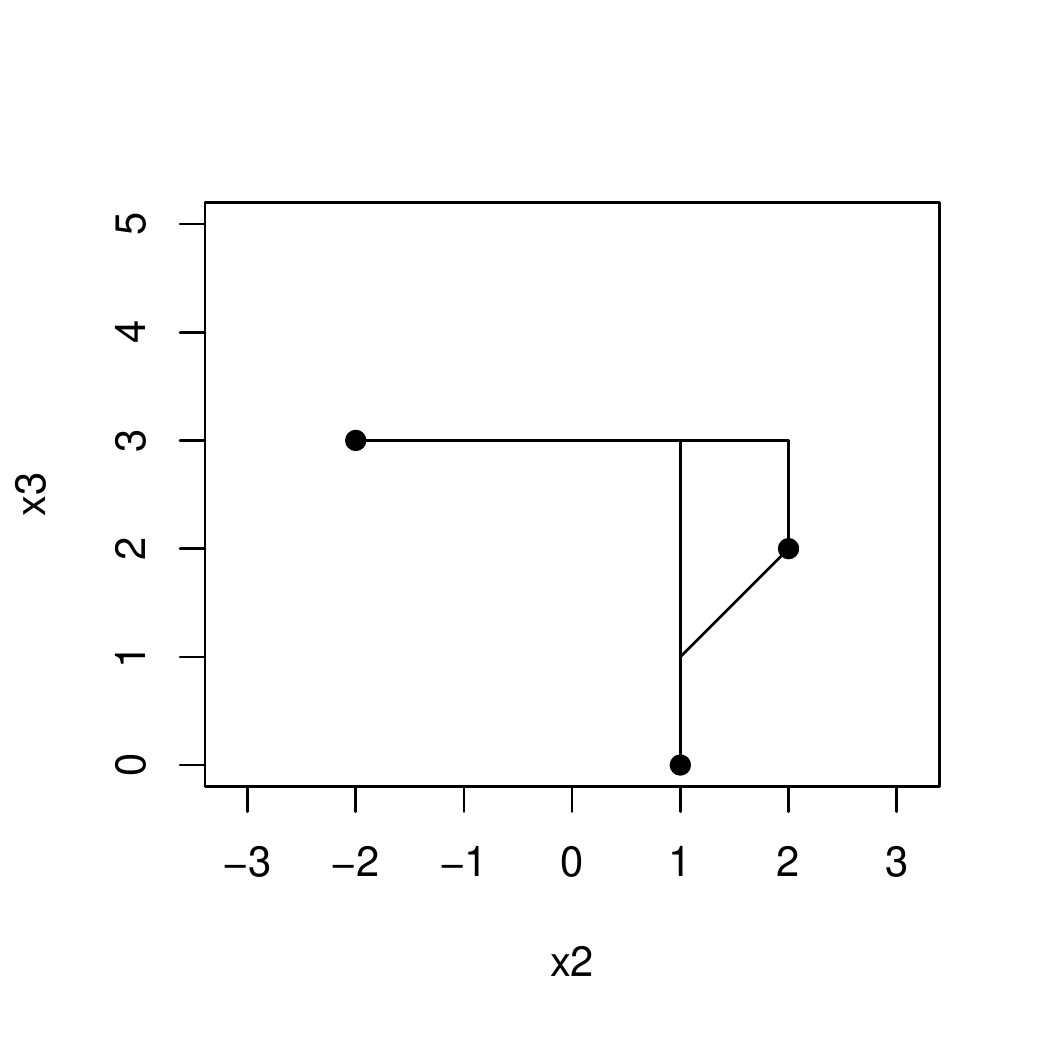}~
 \includegraphics[width=0.45\textwidth]{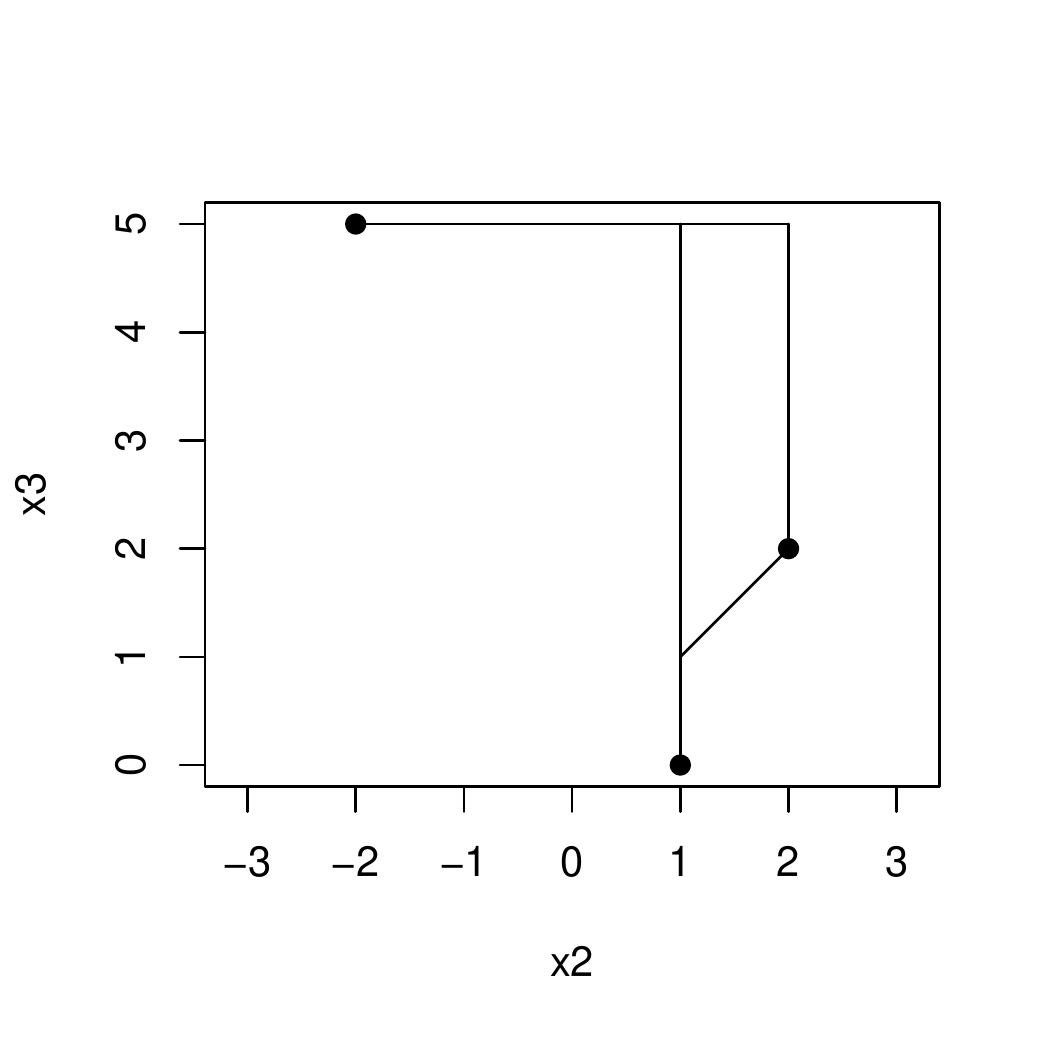}
 \includegraphics[width=0.45\textwidth]{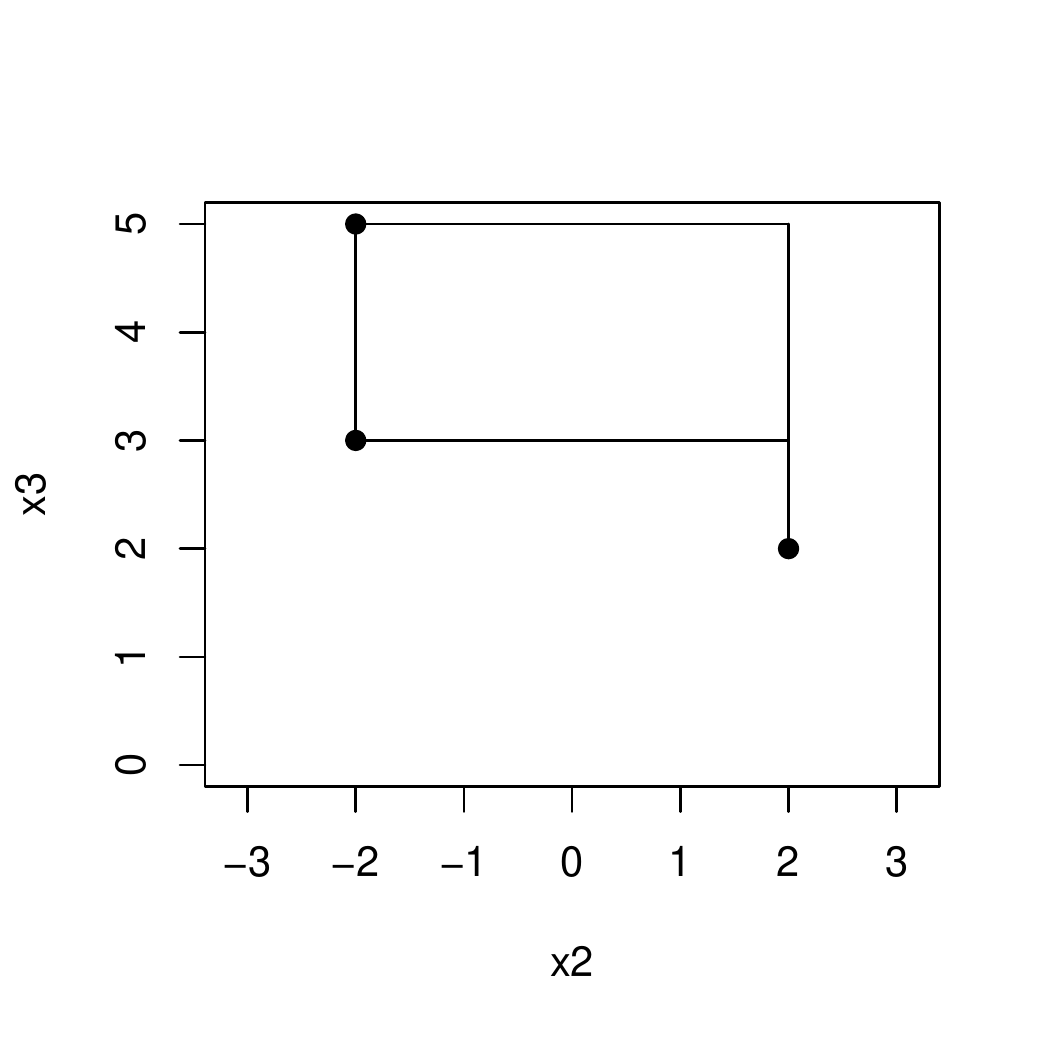}~
 \includegraphics[width=0.45\textwidth]{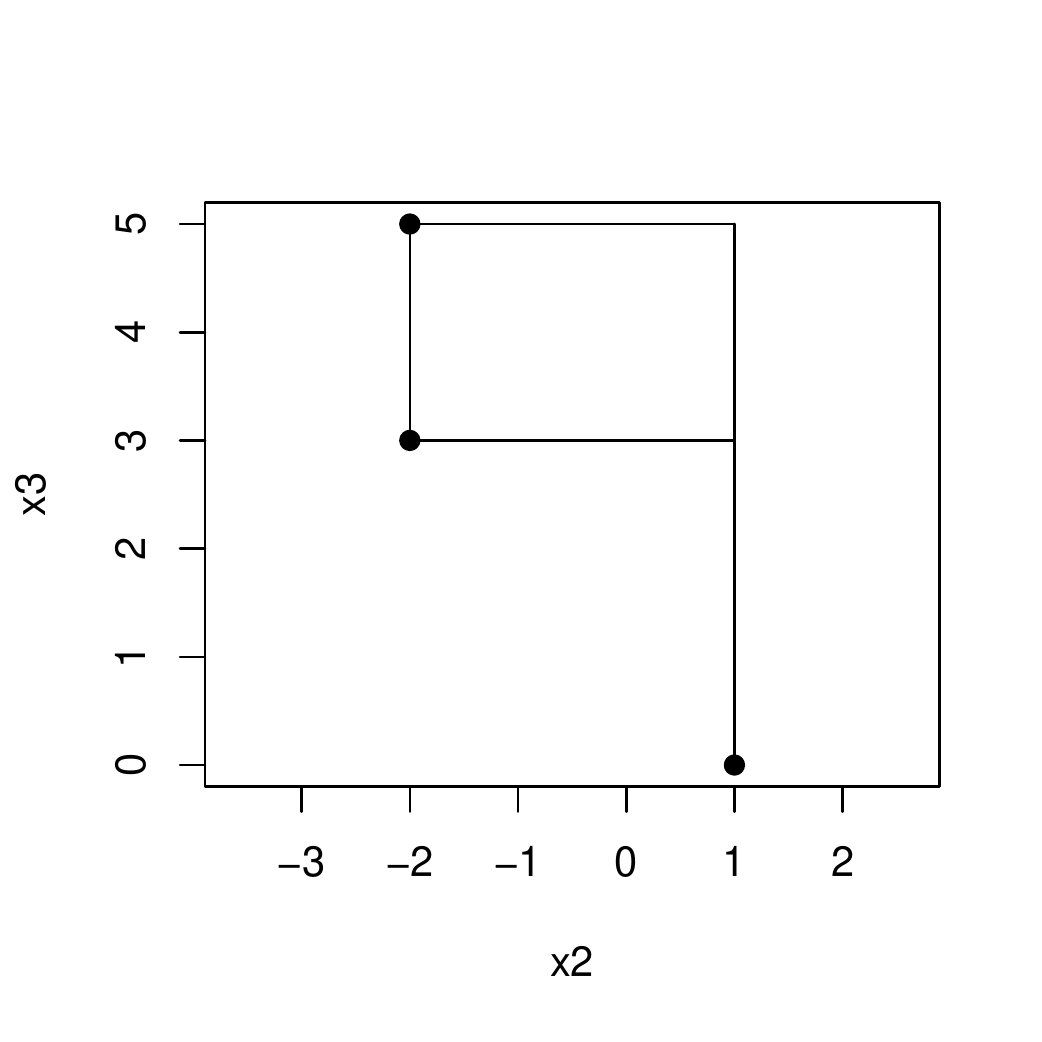}

\caption{The tropical simplices defining the tropical polytope, $P$ for Example~\ref{ex:nc_tpoly}.  $P$ is defined on four vertices so $P$ can be defined by the union of four tropical simplices.}
\label{fig:tpolex1}
\end{figure}

\noindent In Figure~\ref{fig:tpolex1}, the bottom two tropical simplices will yield the $B_R(P)$.  
For brevity we show the $h*-representation(P)$ for the bottom right tropical simplex and linear program formulation using \eqref{eq:maxball}-\eqref{eq:maxball4}.

\begin{align*}
    &\text{$\max_{R,x_0}$} \quad
    R \\
    &\;\text{s.t.} \quad
\;x_0^1-(x_0^2-R)\le 2 \\
&\;x_0^1-(x_0^3-R)\le -3\\
    &\;(x_0^2+R)-x_0^1\le 1\\
    &\;(x_0^3+R)-x_0^1\le 5\\
    &\;(x_0^3+R)-x_0^2 \leq 7\\
    &\;(x_2+R)-x_0^3 \leq 1\\
    &\;x_0^1=0.
\end{align*}
 
\noindent The $B_R(P)$ center point and radius are $x_0=(0,-1,4)$ and $R=1$, respectively.
\begin{figure}[H]
    \centering
    \includegraphics[width=0.44\textwidth]{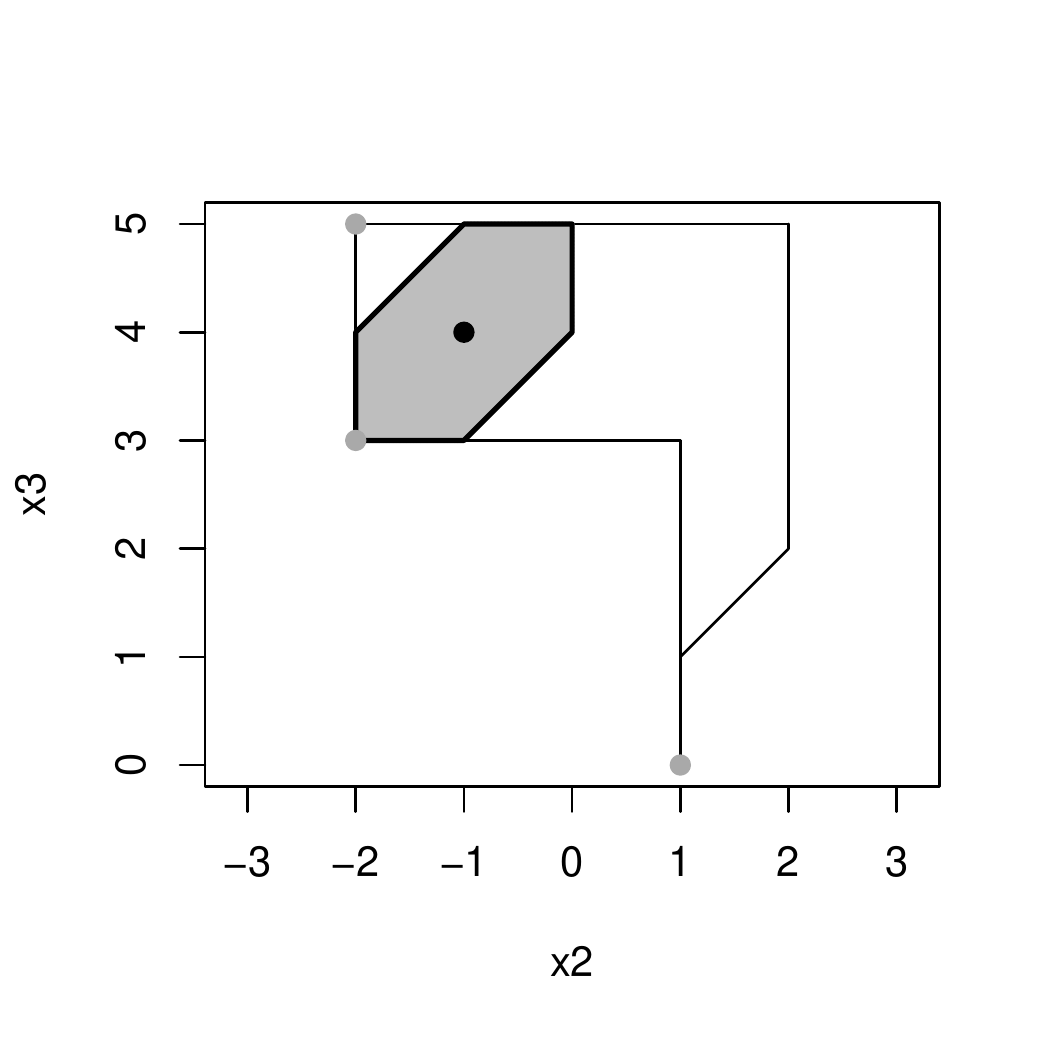}
    \caption{Tropical polytope used in Example~\ref{ex:nc_tpoly} with $B_R(P)$ defined in gray.  The center point, $x_0$ is shown in black.}
    \label{fig:tpolex_max_ins}
\end{figure}
\end{example}

% If the $h^*-representation$ is used to define a polytrope, then each resulting constraint will be needed in the linear program formulation. However, in general this will not be the case.  In Example~\ref{ex:nc_tpoly}, we notice that the fifth and sixth constraints are redundant as the hyperplanes defining the bottom right simplex in Figure~\ref{fig:tpolex1} only intersect the tropical polytope at the vertices $(0,-2,5)$ and $(0,1,0)$, respectively.  These vertices are already defined by the intersection of other hyperplanes in the formulation.  

This final example illustrates the use of the formulation in \eqref{eq:maxball}-\eqref{eq:maxball4} of a tropical simplex which is not a polytrope.

\begin{example}\label{ex:ncpoly2}
    Consider a tropical simplex $P_\Delta=\{(0,0,0,0),(0,1,3,1),(0,1,2,5),(0,2,5,10)\}$.  Let $A$ be a matrix representing the points as column vectors.  We calculate $tdet(A)=14$, and order the columns such that $A_{\sigma(i),i}=a'_{kk}$ where $\sigma(i)=k$. 

    \[A=\left(\begin{array}{cccc}
        0 & 0 &0&0\\
         0& 1&1&2\\
         0&3&2&5\\
         0&1&5&10
    \end{array}\right)\rightarrow 
    \left(\begin{array}{cccc}
        0 & 0 &0&0\\
         0& 1&1&2\\
         0&2&3&5\\
         0&5&1&10
     \end{array}\right)=A'.
    \]

\noindent After constructing $A'$ we add $-A'_{kk}$ to each element of each column $A_k$.

\[A'=\left(\begin{array}{cccc}
        0 & 0 &0&0\\
         0& 1&1&2\\
         0&2&3&5\\
         0&5&1&10
    \end{array}\right)\rightarrow 
    \left(\begin{array}{cccc}
        0 & -1 &-3&-10\\
         0& 0&-2&-8\\
         0&1&0&-5\\
         0&4&-2&0
     \end{array}\right)=\mathbf{m^*}.
    \]
\noindent This results in the following program:

\begin{align*}
    &\text{$\max_{R,x_0}$} \quad
    R \\
    &\;\text{s.t.} \quad
\;x_0^1-(x_0^2-R)\le 0 \\
&\;x_0^1-(x_0^3-R)\le 0\\
&\;x_0^1-(x_0^4-R)\le 0\\
    &\;(x_0^2+R)-x_0^1\le 1\\
    &\;(x_0^3+R)-x_0^1\le 3\\
    &\;(x_0^4+R)-x_0^1\le 10\\
    &\;(x_0^3+R)-x_0^2 \leq 2\\
    &\;(x_0^4+R)-x_0^2 \leq 8\\
    &\;(x_0^2+R)-x_0^3 \leq -1\\
    &\;(x_0^4+R)-x_0^3 \leq 5\\
    &\;(x_0^2+R)-x_0^4 \leq -4\\
    &\;(x_0^3+R)-x_0^4 \leq 2\\
    &\;x_0^1=0.
\end{align*}

\noindent Solving the above linear program results in a center point, $x_0=(0,.5,2,.5)$ with the radius, $R=0.5$.  Figure~\ref{fig:tpolex2} shows $P_\Delta$ and $B_R(P_\Delta)$.

\begin{figure}[H]
 \centering
 \includegraphics[width=0.4\textwidth]{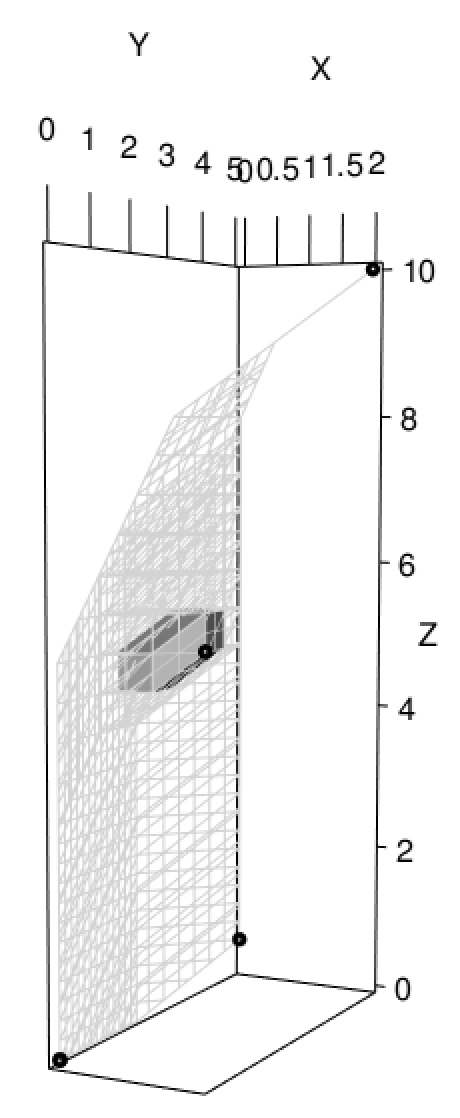} ~
\includegraphics[width=0.38\textwidth]{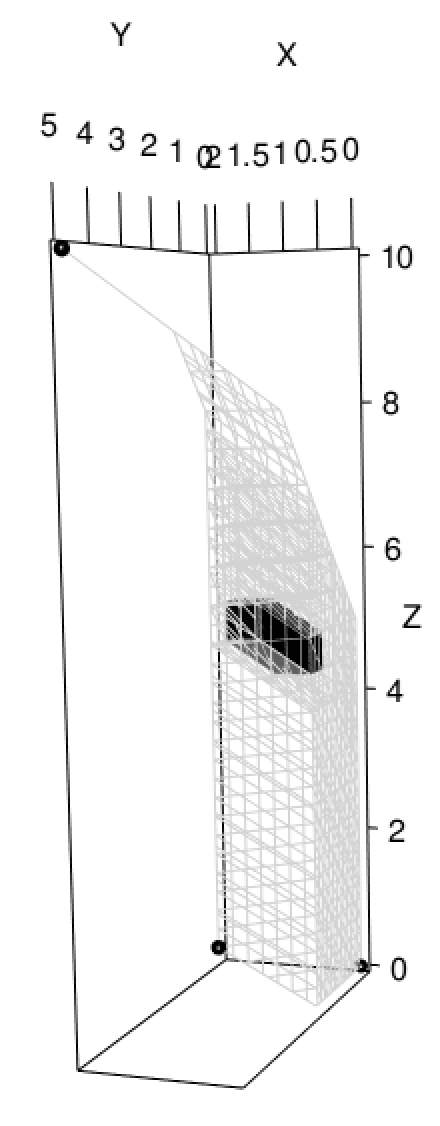}
\caption{$B_R(P_\Delta)$ constructed in Example~\ref{ex:ncpoly2}.}
\label{fig:tpolex2}
\end{figure}
\end{example}

\subsection{Minimum Enclosing Tropical Balls}

Next we illustrate how we can formulate a linear programming problem to find the minimum enclosing ball, $B_r(P)$, for a tropical polytope $P$.  We seek $y_0\in \mathbb{R}^e/\mathbb {R}\mathbf{1}$ to be the center point of $B_r(P)$ which minimizes $d_{tr}(y_0,v_i)$ for each $v_i\in V$ where $P=\tconv(V)$, such that $P\subset B_r(P)$.  The largest minimized $d_{tr}(y_0,v_i)$ for all $v_i \in V$ is the radius, $r$, of $B_r(P)$.  Com\u{a}neci et al. showed in~\citep{transport}, that for any two points, $x,y \in \mathbb{R}^e/\mathbb{R}\mathbf{1}$, $d_{tr}(x,y)$ can be rewritten in the following way:

\begin{equation}\label{eq:maxhyperp2}
    d_{tr}(x,y)=\text{$\max_i$}(x-y)-\text{$\min_j$}(x-y)=\text{$\max_{\substack{i,j\\i\neq j}}$}(x_i-y_i-x_j+y_j).
\end{equation}

\noindent We note that this is equivalent to the following formulation shown in~\citep{joswigBook},

\begin{equation}\label{eq:maxhyperp3}
    d_{tr}(x,y)=\text{$\max_{1\leq i<j\leq e}$}(|x_i-y_i-x_j+y_j|).
\end{equation}

Using the result from~\eqref{eq:maxhyperp2}, we can formulate the following optimization problem for a given $V$ to find $y_0$ and $r$ of $B_r(P)$.

\begin{align}\label{eq:minball1}
    \text{$\min_{y_0}$  }& \text{$\max_{\substack{i\in[|V|] \\j,k\in[e]\\j\ne k}}\,$}\left(v_{ij}-y_0^j-v_{ik}+y_0^k\right)\\
    &\text{s.t. } v_i\in V.
\end{align}

Because we are minimizing a maximum of a set of linear functions, the objective function in~\eqref{eq:minball1} may be further manipulated to formulate a linear program to solve for $y_0$ by replacing the maximum function with the variable $r$, representing the radius of $B_r(P)$. This formulation is shown in~\eqref{eq:minball}-\eqref{eq:minball3} as

\begin{align}%\label{eq:minball}
    \text{$\min_{r,y_0}$} \quad
    r& \label{eq:minball} \\ % \label{eq:minball11}
    \text{s.t.} \quad 
    &v_{i,j}-y_0^j-v_{i,k}+y_0^k\le r \quad \label{eq:minball2} \\
    &\indent \forall\;i\in[|V|];\,j,k\in\;[e],\;j\neq k, \nonumber \\
    % &y_0^j-v_{i,j}-y_0^k+v_{i,k}\le r \quad\\\nonumber
    % &\indent \forall\;i,j,k\in\;\{1\ldots e\},\;j\neq k,\;v\in S\\\label{eq:minball3}
    &r \geq 0. \label{eq:minball3}
\end{align}

\noindent Importantly, $y_0$ need not be in $P$ to construct $B_r(P)$.  Because~\eqref{eq:maxhyperp2} is equivalent to~\eqref{eq:maxhyperp3}, either will result in the similar constraint construction (See Proposition 25 in~\citep{LSTY}).
To give a lower bound on the radius of $B_r(P)$, we arrive at the following proposition.

\begin{proposition}\label{thm:ext}
For a minimum enclosing tropical ball of a tropical polytope $P$,  $B_r(P)$ with the radius $r$ satisfies the following inequality:
\begin{equation}\label{eq:t_ballmin1}
    r \geq \max \limits_{i,j} \frac{d_{tr}(v_i,v_j)}{2},
\end{equation}
where $v_i,v_j \in V$ with $V$ being the vertex set of $P$.
\end{proposition}

\begin{proof}

Let $V'$ be a minimum vertex set for a tropical polytope $P$, where $P=\tconv(V')$ and let $W'$ be the minimum vertex set where $B_r(P)=\tconv(W')$. We are interested in finding the minimum enclosing ball, $B_r(P)$.  Recall that for any $x,y \in P$ 

\[d_{tr}(x,y)\leq \max_{i,j}d_{tr}(v_i,v_j)\] 

\noindent where $v_i,v_j \in V'$ and $i,j\in[|V'|]$.   Therefore, 

\[\max_{i,j}d_{tr}(v_i,v_j)\leq \min_{k,l}d_{tr}(w_k,w_l),\] 

\noindent where $w_k,w_l \in W'$ and $k,l\in[|W'|]$.  Considering $B_r(P)$ and $W'$,   
\[d_{tr}(w_k,w_l)=2r\;\;\forall\; k,l\in[|W'|],\]
\noindent by definition of a tropical ball where $r$ is the radius of the the tropical ball.  Therefore, since $B_r(P)$ encompasses $P$,
\[\max_{i,j}\frac{d_{tr}(v_i,v_j)}{2}\leq r.\]
\end{proof}
   
\begin{example}\label{ex:ex1}
Consider the tropical polytope, $P$, generated by $V=\{(0, 0, 0), \, (0, 1, 0), \,$\\$ (0, 0, 1)\}$ in $\mathbb R^3 \!/\mathbb R {\bf 1}$ . The maximum pairwise tropical distance between the vertices is \\$d_{tr}((0,1,0),(0,0,1))=2$.  Therefore, $r\geq1$  (Figure \ref{fig:2D_Ball} (left)).  
\end{example}

\begin{example}\label{ex:ex3}
Consider the tropical polytope, $P$, generated by $V=\{(0, 0, 0), \, (0, 2, 5), \,$\\$ (0, 3, 1)\}$ in $\mathbb R^3 \!/\mathbb R {\bf 1}$ . The maximum pairwise tropical distance between the vertices is $d_{tr}((0,0,0),$\\$(0,2,5))=d_{tr}((0,3,1),(0,2,5))=5$.  Therefore $r\geq2.5$  (Figure \ref{fig:2D_Ball} (right)).  
\end{example}
\begin{example}\label{ex:ex2}
Consider the tropical polytope, $P$, generated by $V=\{(0, -2, 3), \, (0, -2, 5), \,$\\$ (0, 2, 2), \, (0,1,1)\}$ in $\mathbb R^3 \!/\mathbb R {\bf 1}$ . The maximum pairwise tropical distance between the vertices is $d_{tr}((0,-2,5),(0,1,0))=8$.  Therefore $r\geq4$  (Figure \ref{fig:2D_Ball} (bottom)).  
\end{example}

\begin{figure}[H]
    \centering
    \includegraphics[width=0.5\textwidth]{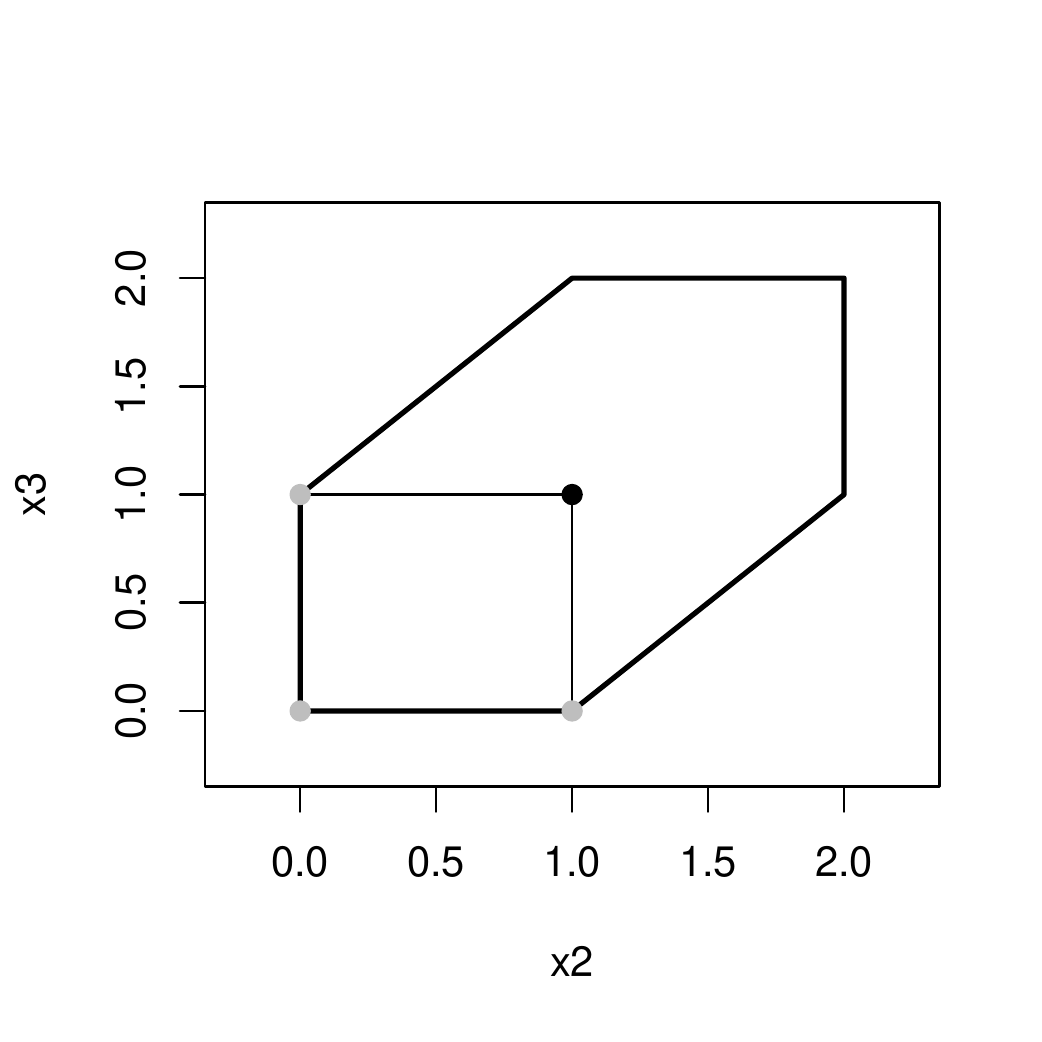}~
    \includegraphics[width=0.5\textwidth]{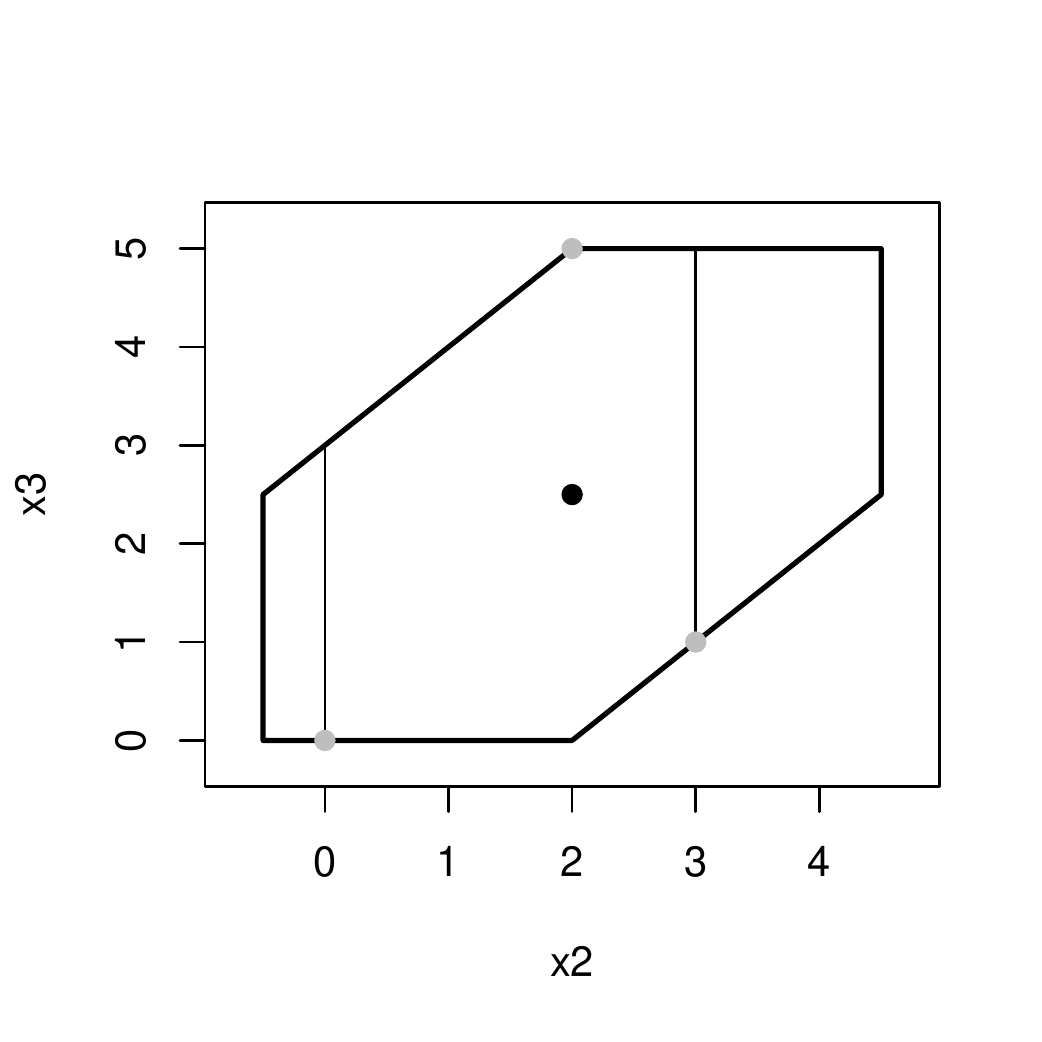}
     \includegraphics[width=0.5\textwidth]{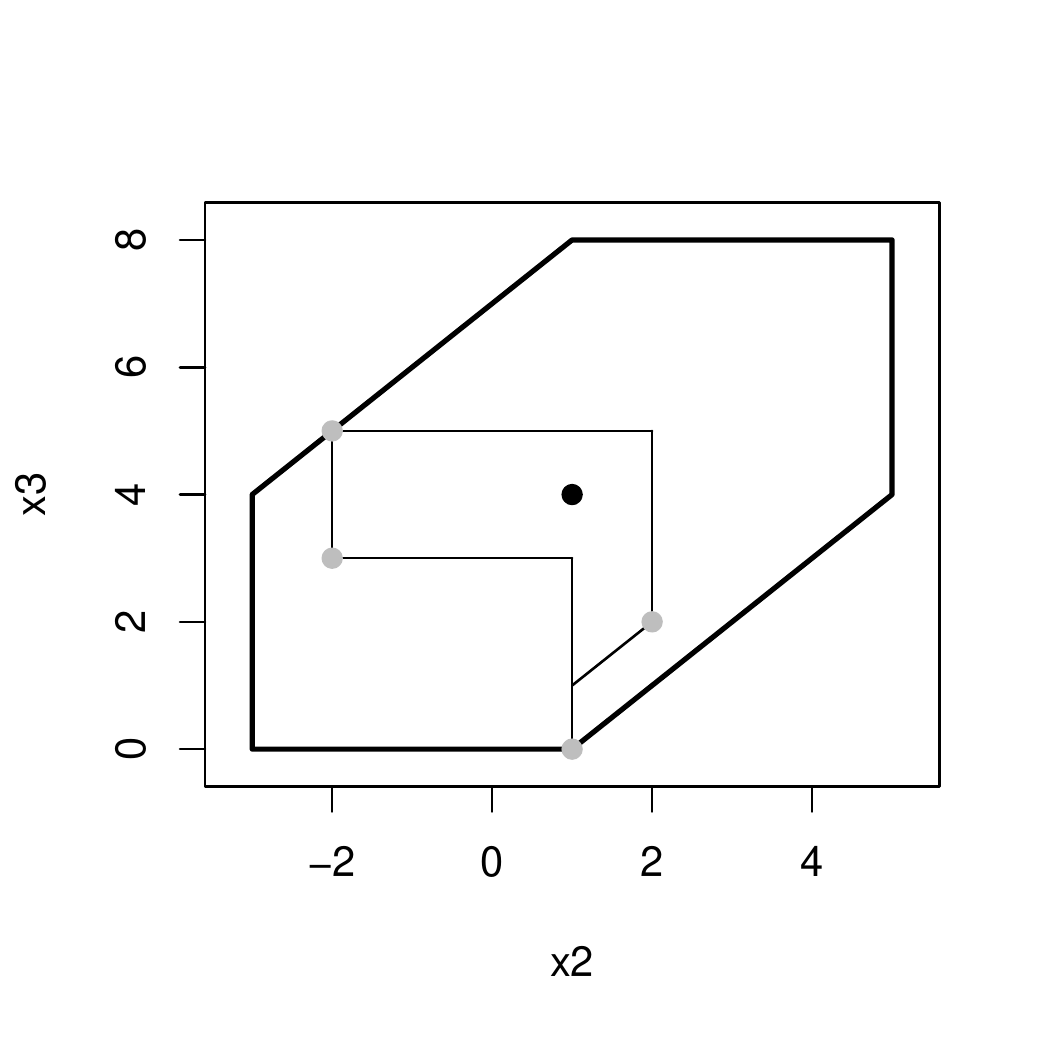}
    \caption{Results for Example~\ref{ex:ex1},~\ref{ex:ex2}, and~\ref{ex:ex3}.  Tropical polytopes are shown with the defining vertices in gray and associated $B_r(P)$ represented by the black lines.  Black points represent the center of the $B_r(P)$.}
    \label{fig:2D_Ball}
\end{figure}
\newpage
Employing~\eqref{eq:minball}-\eqref{eq:minball3} in each of the previous examples shows that equality holds when using the result from Proposition~\ref{thm:ext}. To show that equality does not always hold in~\eqref{eq:t_ballmin1}, we provide the following example of a tropical polytope in $\mathbb R^3 \!/\mathbb R {\bf 1}$.

\begin{example}\label{ex:ex4}
Consider the tropical polytope, $P$, generated by four vertices $(0, 0, 0,0), \, (0, 2, 5,0), \,$\\$ (0, 3, 1,0),$ $ \, (0,2,5,5)$ in $\mathbb R^4 \!/\mathbb R {\bf 1}$ . The maximum pairwise tropical distance between the vertices is $d_{tr}((0,3,1,0),(0,2,5,5))=6$.  However, solving~\eqref{eq:minball}-(15) yields $r=3.333$ with $y_0=(0,0,2,3.333,1.667)$  (Figure \ref{fig:3D_Ball}).  
\end{example}

\begin{figure}[H]
    \centering
    \includegraphics[width=0.44\textwidth]{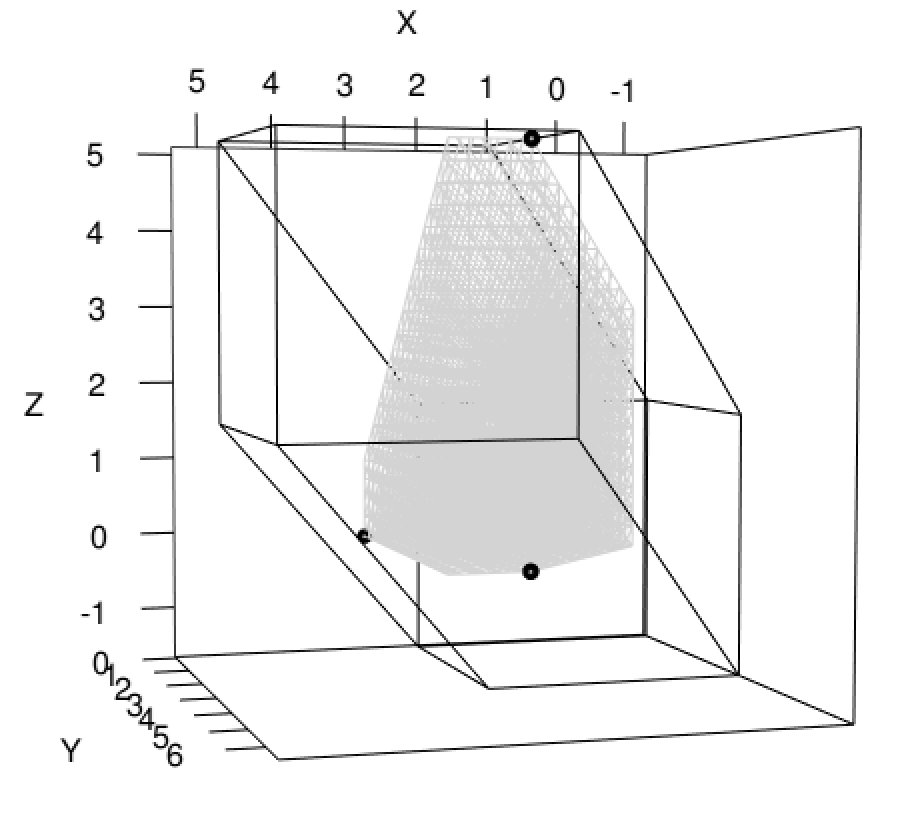}~
    \includegraphics[width=0.44\textwidth]{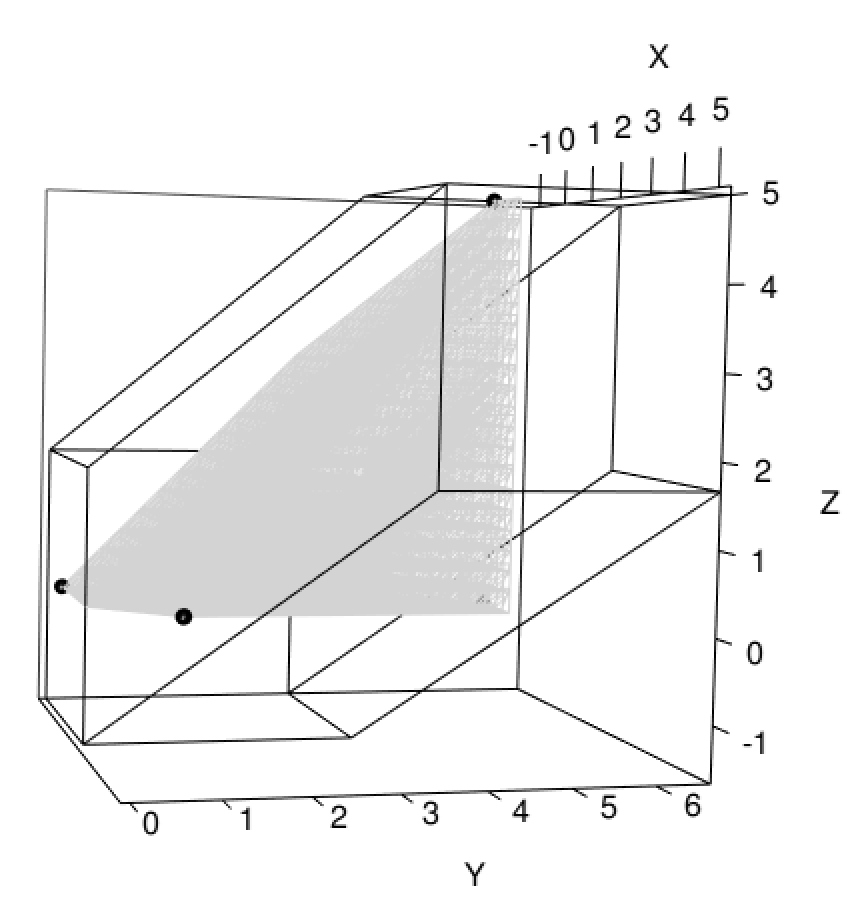}
     \includegraphics[width=0.44\textwidth]{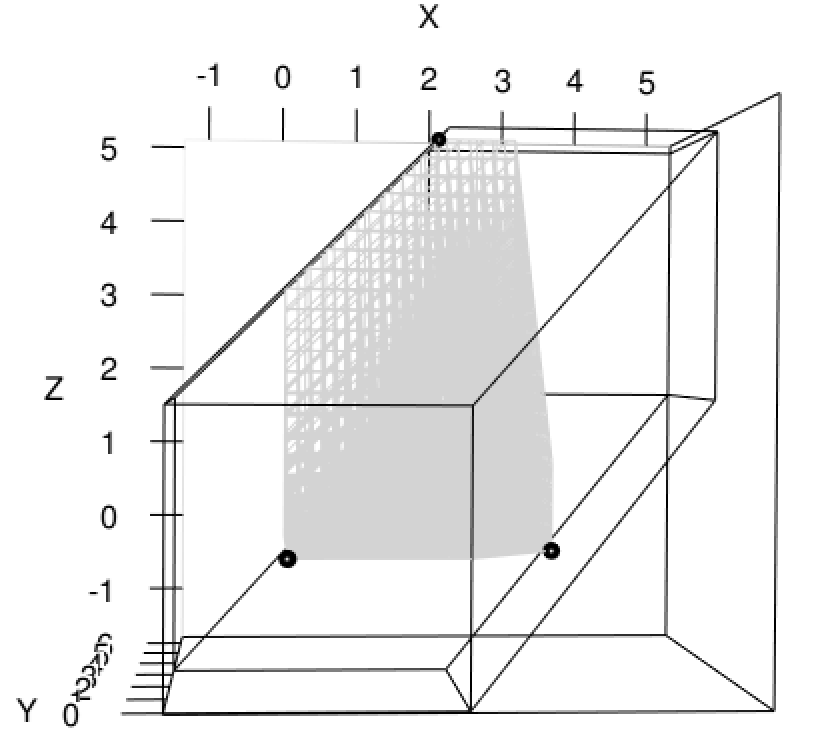}
    \caption{Results for Example~\ref{ex:ex4}.  The tropical polytope is shown in gray with its defining vertices in black. $B_r(P)$ is represented by the black lines.}
    \label{fig:3D_Ball}
\end{figure}

Knowing the $B_r(P)$ and $B_R(P)$ providess information about $P$ which we investigate further in Section~\ref{sec:Vol_est}.

\section{Estimating the Volume of a Tropical Polytope}\label{sec:Vol_est}

Estimating the volume of classical polytopes has wide application in math and science.  As shown in many publications, finding exact volume of any convex body is challenging, especially in higher dimensions~\citep{VOL_EST,  chev_vol}.  A number of methods have been developed to estimate the volume of classical polytopes to include the use of Ehrhart theory as well as MCMC methods~\citep{CV,EHR1}. In this section, we introduce a method to estimate the Euclidean volume of a tropical polytope using a HAR sampler with the tropical metric, an analogue of the method introduced in~\citep{CV} for classical polytopes.

\begin{remark}
As shown in~\citep{Gaubert}, computing the volume of a tropical polytope, $Vol(P)$, is NP-hard (see Theorem 5.5 and Corollary 5.6 in \citep{Gaubert}). Further, if a tropical polytope is described by inequalities, calculating  $Vol(P)$ is \#P-hard \citep[Theorem 8.1]{Gaubert}.  Additionally, we know that there is no polynomial time algorithm to approximate $Vol(P)$ unless $P = NP$ \citep[Theorem 7.8]{Gaubert}.   Therefore, in general estimating $Vol(P)$ is a hard problem.
\end{remark}
 % However,  we can use $B_r(P)$ to estimate $Vol(P)$, by applying a HAR sampler with the tropical metric, an analogue of a HAR sampler described in~\citep{smith} for classical polytopes.  Additionaly, we can compute upper and lower bounds on $Vol(P)$ from $B_r(P)$ and $B_R(P)$.  %While simply implemented, volume estimation remains a difficult problem especially as dimensions increase.

\subsection{Volume of a Tropical Ball in Tropical Projective Space}

In this section we apply $B_r(P)$, whose volume can be computed analytically, to the problem of estimating the volume of a given tropical polytope.  Ehrhart showed that the volume of a classical polytope can be determined by the leading coefficient of a rational quasi-polynomial for counting the number of lattice points inside of a dilated polytope~\citep{EHR1}.  Because a tropical ball, $B_l(x_0)$, is classically convex, we can calculate $Vol(B_l(x_0))$ by calculating the Ehrhart quasi-polynomial and determine its lead coefficient.  Even more helpful is how a tropical ball, $B_l(x_0) \in \mathbb{R}^e/\mathbb{R}\mathbf{1}$, essentially consists of $e$ hypercubes where each edge of the hypercube is of length, $l$.  This structure allows us to calculate $Vol(B_l(x_0))$ using the following equation:

\begin{equation}\label{eq:t_ball_vol}
    Vol(B_l(x_0))=e \times l^{(e-1)}.
\end{equation}

Equation~\eqref{eq:t_ball_vol} was shown  in  \citep[Corollary 3.3]{Gaubert}  and provides a straightforward calculation of $Vol(B_l(x_0))$. Table~\ref{tab:tball_volumes} shows different values of $Vol(B_l(x_0)$ of radius, $l=2$ and $l=4$, with increasing dimension.

\begin{table}[H]
    \centering
    \begin{tabular}{|c|c|c|}
    \hline
        $n$ & $Vol(B_2(x_0))$& $Vol(B_4(x_0))$\\\hline \hline
        3 &  6& 48 \\\hline
    4 &  32&256 \\\hline
    5 &  80&1280 \\\hline
    6 &  192 &6144\\\hline
    10 & 5120& $\approx 2.62 \times 10^6$\\\hline
    \end{tabular}
    \caption{Volume of $B_l(x_0) \in \mathbb{R}^n/\mathbb{R}\mathbf{1}$ with increasing $n$ for $l=2$ and $l=4$.}\label{tab:tball_volumes}
\end{table}

\subsection{Volume Estimation by Sampling from a Minimum Enclosing Ball}

We estimate the volume of a given tropical polytope via sampling points from $B_r(P)$ enclosing $P$ and then determining the proportion, $p$, of points that also fall inside of $P$.  Multiplying $Vol(B_r(P))$ by $p$ gives an estimate of $Vol(P)$. Knowing $B_r(P)$ for any tropical polytope, $P$, provides the bounds of $Vol(P)$.  

\begin{proposition}\label{prop:ub_bol}
Let $P$ be a tropical polytope.  Then $Vol(P)\leq Vol(B_r(P))$. Specifically, $Vol(B_r(P))$ is an upper bound on $Vol(P)$.
\end{proposition}

\begin{proof}
By definition, $B_r(P)$ encloses $P$.  Therefore, $Vol(P)$ must be less than or equal to $Vol(B_r(P))$, establishing an upper bound.
\end{proof}

Likewise, we can establish a lower bound for any tropical polytope with a $Tr_{e-1}(P)$.  In~\citep{Gaubert}, the authors show the lower bound for a polytrope is the volume of the maximum inscribed ball (See Corollary 3.6).  The following proposition is an extension to the general case.

\begin{proposition}\label{prop:lb_bol}
Let $P$ be any tropical polytope containing a $Tr_{e-1}(P)$ where $e\geq 3$.  The $B_R(P)$ contained in $Tr_{e-1}(P)$ provides a lower bound on the volume of $P$.  Specifically,

\[Vol(P)\geq Vol(B_R(P))\].

\end{proposition}

\begin{proof}
The $Tr_{e-1}(P)$, of $P$ represents a tropical polytope of dimension $(e-1)$~\citep{TVOL_2}.  Any tropical polytope, $P$, where $dim(P)=(e-1)$ contains a $B_R(P)$.   Therefore, $B_R(P)$ of the $(e-1)$-trunk, $B_R(Tr_{e-1}(P))=B_R(P)$. Therefore,
\[Vol(B_R(P))=Vol(B_R(Tr_{e-1}(P))\leq Vol(P)).\]
%\noindent achieving the result.  
\end{proof}

While Propositions~\ref{prop:ub_bol} and~\ref{prop:lb_bol} provide bounds on $Vol(P)$, the range between these bounds can be very large.  

Next we summarize our volume estimation technique.  Given a tropical polytope $P\in\mathbb{R}^e/\mathbb{R}\mathbf{1}$, calculate $B_r(P)$.  Using the HAR sampler described in~\ref{app:min_plus_HAR} (see Algorithm~\ref{alg:HAR_extrapolation4}), sample points from $B_r(P)$.  Determine the proportion $p$, of sampled points from $B_r(P)$ that also fall in $P$.  Estimate $Vol(P)$ by the following calculation with $\cdot$ representing classical multiplication,

\begin{equation}
    Vol(P)\approx p\cdot Vol(B_r(P)).
\end{equation}

\noindent Algorithm~\ref{alg:Vol_Est} below describes  these steps to estimates the volume of $P$ by sampling from $B_r(P)$.

\begin{algorithm}[H]
\caption{Volume Estimation of $P$ by Sampling from $B_r(P)$.} \label{alg:Vol_Est}
\begin{algorithmic}
\State {\bf Input:} Tropical polytope $P:=\tconv(v_1, \ldots, v_s)$; a minimum enclosing ball $B_r(P)$ calculated using~\eqref{eq:minball}; an initial point $x_0 \in\mathbb{R}/\mathbb{R}\mathbf{1}$ and desired number of points $I$.
\State {\bf Output:} Volume estimate of $P$.
\State Set $C=0$
\For{$k= 0, \ldots , I-1$,}
\State Generate random point $x_{k+1}$ using Algorithm~\ref{alg:HAR_extrapolation4} in~\ref{app:min_plus_HAR}.
\If{$x_{k+1}\in P$} 
\State$C=C+1$.
\EndIf
\EndFor 
\State Set $p=\frac{C}{I}$
\State Set $Vol(P)=p\cdot Vol(B_r(P))$\\
\Return $Vol(P)$.
\end{algorithmic}
\end{algorithm}

Because Algorithm~\ref{alg:Vol_Est} samples from $B_r(P)$ in order to estimate $Vol(P)$, sampling more points from $B_r(P)$ will yield better results.  Therefore, the number of points to sample is related to $Vol(B_r(P))$ relative to $Vol(Tr_{e-1}(P))$.  One way to evaluate this relationship is by comparing the ratio between the radii of $B_R(P)$ and $B_r(P)$ which is akin to a similar method described in~\citep{CV} for classical polytopes.  A small ratio suggests that $Vol(Tr_{e-1}(P))$ is relatively small compared $Vol(B_r(P))$, requiring more sampled points to obtain a better estimate  of $Vol(P)$ leading to the next proposition.  

\begin{proposition}
Given a minimum enclosing ball, $B_r(P)$, and maximum inscribed ball, $B_R(P)$, the rate, $A(P)$, of sampled points from $B_r(P)$ falling inside $P$ is bounded below by,
\begin{equation}
    A(P)\geq \left(\frac{R}{r}\right)^{(n-1)}.
\end{equation}
\end{proposition}
\begin{proof}\label{prop:ac_rate}
Consider two tropical balls, $B_l(x_0)$ and $B_k(y_0)$, in $\mathbb{R}^n/\mathbb{R}\mathbf{1}$ such that $B_l(x_0)\subset B_k(y_0)$.  Recall that $Vol(B_l(x_0))$ and $Vol(B_k(y_0))$ are found using~\eqref{eq:t_ball_vol}. Let,
\begin{align}
    \frac{Vol(B_l(x_0))}{Vol(B_k(y_0))}&=\frac{n*l^{n-1}}{n*k^{n-1}}\\
    &=\frac{l^{n-1}}{k^{n-1}}\\
    &=\left(\frac{l}{k}\right)^{(n-1)}.
\end{align}
Therefore we achieve the result.
\end{proof}

Proposition~\ref{prop:ac_rate} helps provide insight to the difficulty in estimating $Vol(P)$ using Algorithm~\ref{alg:Vol_Est}.  Further,  as $dim(P)$ increases, $Vol(B_r(P))$ increases as shown in Table~\ref{tab:tball_volumes}, for a constant $r$.

\begin{example}\label{ex:rat_comp}
    Consider the tropical polytope, $P$, defined in Example~\ref{ex:ncpoly2} and shown in Figure~\ref{fig:tpolex2}.  The maximum inscribed ball for $P$ has a radius, $R=0.5$, indicating that $Vol(B_R(P))=0.5$.  The minimium enclosing ball of $P$ has a radius, $r=5$, resulting in $Vol(B_r(P))=500$.  Employing Algorithm~\ref{alg:Vol_Est} in this case will involve sampling $B_r(P)$ where a $Vol(B_r(P))=500$ to estimate $Vol(P)$ that has a volume closer to $0.5$.  
\end{example}

Next, we offer some results using Algorithm~\ref{alg:Vol_Est} applied to some of the previously defined tropical polytopes.

\begin{example}[Example~\ref{ex:rat_comp} cont'd]\label{ex:polytope_vol}
    Consider the tropical polytope from Example~\ref{ex:ncpoly2}.  $B_r(P)$ has a radius of $r=5$ meaning that $Vol(B_r(P))=500$.  Table~\ref{tab:Vol_est1} shows the volume estimate using Algorithm~\ref{alg:Vol_Est} using increasing sample sizes. Here $I$ is the sample size and $C$ is the number of sampled points that fall in $P$.
 \begin{table}[H]
    \centering
    \begin{tabular}{|l|l|l|}
    \hline
        \multicolumn{1}{|c|}{$I$} &\multicolumn{1}{|c|}{$C/I$} &\multicolumn{1}{|c|}{$Vol(P)$}\\\hline \hline
        1000 &  0.003&1.5 \\\hline
    10000 & 0.0049 & 2.45\\\hline
    100000 &  0.00529& 2.645\\\hline
    \end{tabular}
    \caption{Volume estimate of $P \in \mathbb{R}^4/\mathbb{R}\mathbf{1}$ with varying number of sampled points, $I$, from $B_r(P)$.}
    \label{tab:Vol_est1}
\end{table}
    
\end{example}

\begin{example}[Example~\ref{ex:nc_tpoly}, cont'd]
    We apply Algorithm~\ref{alg:Vol_Est} to $P$ from Example~\ref{ex:nc_tpoly}, varying the number of sampled points taken from the $B_r(P)$.  Recall that $r=4$ for $B_r(P)$ meaning that $Vol(B_r(P))=48$. Table~\ref{tab:Vol_est2} shows the estimate of $Vol(P)$ and Figure~\ref{fig:ex_vol2d} illustrates the point dispersion for $I=1000$ and $I=10000$.

\begin{table}[H]
    \centering
    \begin{tabular}{|l|l|l|}
    \hline
        \multicolumn{1}{|c|}{$I$} &\multicolumn{1}{|c|}{$C/I$} &\multicolumn{1}{|c|}{$Vol(P)$}\\\hline \hline
        1000 &  0.168&8.064 \\\hline
    10000 & 0.1993 & 9.564\\\hline
    100000 &  0.1965& 9.4392\\\hline
    \end{tabular}
    \caption{Volume estimate of $P \in \mathbb{R}^3/\mathbb{R}\mathbf{1}$ with varying number of sampled points, $I$, from $B_r(P)$.}
    \label{tab:Vol_est2}
\end{table}
\begin{figure}[H]
    \centering
    \includegraphics[width=0.44\textwidth]{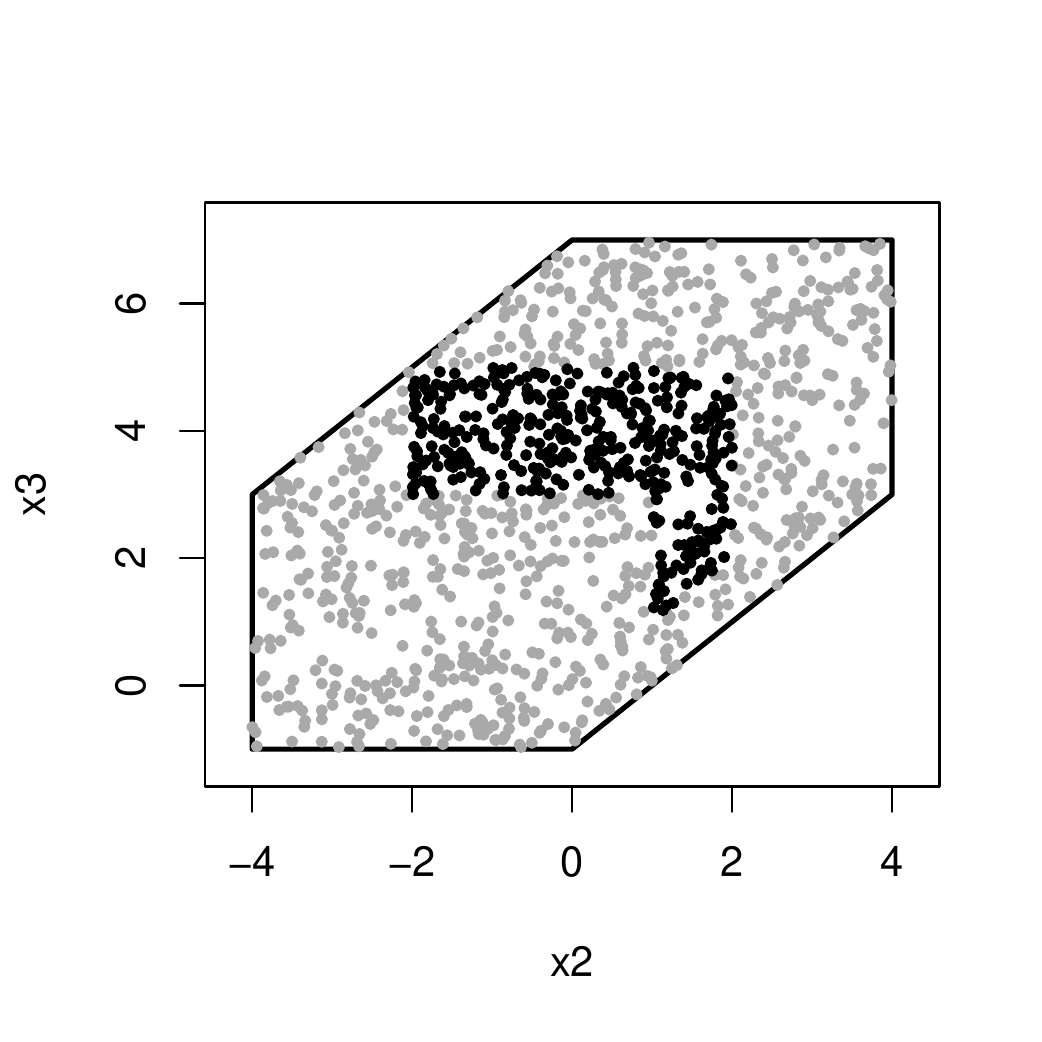}~
    \includegraphics[width=0.44\textwidth]{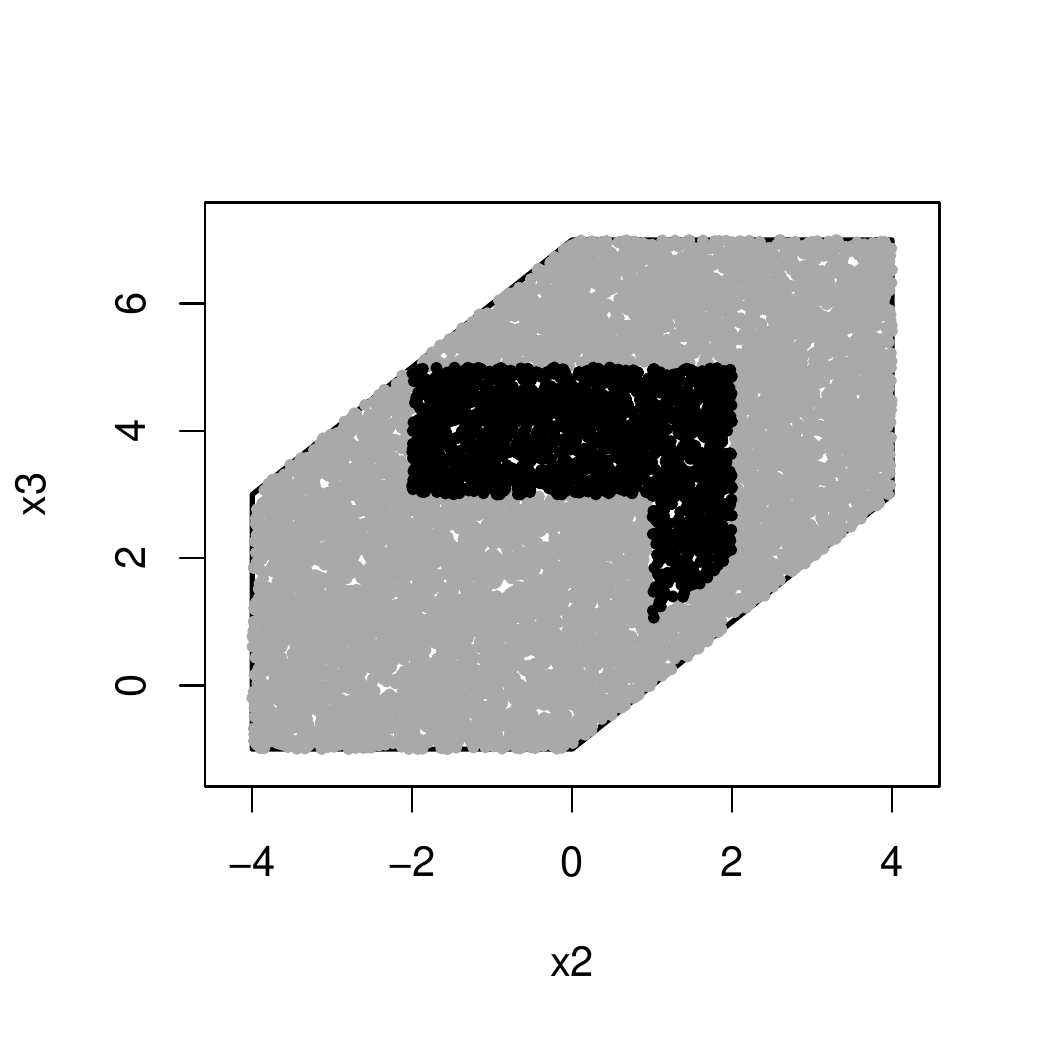}
    \caption{Results for Example 4.6 sampling 1,000 points (left) and 10,000 points (right). Black points indicate those that fall inside of $P$.}
    \label{fig:ex_vol2d}
\end{figure}
\end{example}

\subsection{Rounding a Tropical Polytope}
In this section we consider the idea of ``rounding'' a tropical polytope in order to reduce the sample size needed to estimate $Vol(P)$. In~\citep{CV}, rounding a classical polytope means applying a linear transformation to the polytope in near-isotropic position.  This linear transformation reduces the difference between radii of the Euclidean maximum inscribed ball and the minimum enclosing ball for a classical polytope.  This reduces the sample size required to estimate the volume of a given classical polytope~\citep{CV}.

In this section, we propose an analogue of this linear transformation to a tropical polytope $P$. To do so, first, we identify a minimum enclosing tropical ball that only contains $Tr_{e-1}(P)$ denoted as $B_k(Tr_{e-1}(P))$ by removing all {\em (e-2)-tentacles} of $P$ such that only $Tr_{e-1}(P)$ remains.  Since $Tr_{e-1}(P)\subseteq P$, we have $B_k(Tr_{e-1}(P)) \subseteq B_r(P)$ where $k\leq r$.  Therefore, we have \[Vol(B_k(Tr_{e-1}(P)))\leq Vol(B_r(P)).\]  
Thus, by removing all {\em (e-2)-tentacles} of $P$, the minimum enclosing tropical ball, $B_k(Tr_{e-1}(P))$, has smaller volume, requiring a smaller sample size to estimate $Vol(P)$.
%Finding $B_k(Tr_{e-1}(P))$ allows us to use Algorithm~\ref{alg:HAR_extrapolation4} to sample from a minimum enclosing tropical ball of smaller volume thereby reducing the number of sample points necessary to obtain $Vol(P)$.  

Here we will focus only on a tropical simplex $P_\Delta$, that is not a polytrope but possesses a $Tr_{e-1}(P)$ (see Figure~\ref{fig:tpolex2}). For any $P_\Delta$ possessing $Tr_{e-1}(P_\Delta)$, $Tr_{e-1}(P_\Delta)$ is classically convex.  For the purposes of this paper, we begin with a vertex representation of $P_\Delta$ denoted as $v-representation(P_\Delta)$.  

% In our case, this is made more manageable since we only focus on $Tr_{e-1}(P)$ which will have at most $\binom{2e}{e}$ pseudo-vertices~\citep{JoswigKulas+2010+333+352}. 

Algorithm~\ref{alg:rounding} illustrates how to find $B_k(Tr_{e-1}(P_\Delta))$.  The algorithm involves enumerating {\em pseudo-vertices} which is known to be NP-hard for classical polyhedra with its hyperplane-representation~\citep{Vertex_ENUM}.  

\begin{definition}[Pseudo-vertex (See~\citep{JoswigKulas+2010+333+352})]\label{def:pseudovert}
    For a given $\mathcal{A}^{min}(V)$, a pseudo-vertex of a $\max$-tropical polytope $P$, is any point, $x\in P$ coincident with the intersection of $(e-1)$ or more min-tropical hyperplanes. 
\end{definition}

To enumerate the pseudo-vertices of $Tr_{e-1}(P_\Delta)$, we first construct the $h^*-representation(P_\Delta)$ of  $P_\Delta$.  We then apply a {\em double-description} algorithm described in~\citep{Avis} and implemented by Fukuda in \texttt{cddlib}~\citep{cddlib} where a {\em double-description} is defined as follows,

\begin{definition}[Double Description Methods (See~\citep{TDDM})]
    A {\em double description method} is an algorithm that allows the computation of the vertex representation of a polyhedron that is defined by inequalities.
\end{definition}

When using the {\em double-description} the $h^*-representation(P_\Delta)$ is the input.  Because the {\em double-description} algorithm is used to enumerate the vertices of classical polytopes described in hyperplane-representation, only the pseudo-vertices of the $Tr_{e-1}(P_\Delta)$ will be enumerated since $Tr_{e-1}(P_\Delta)$ is classically convex.  By then applying~\eqref{eq:minball}-~\eqref{eq:minball3} we can construct $B_k(Tr_{e-1}(P_\Delta))$. 

\begin{algorithm}[H]
\caption{Rounding a tropical simplex, $P_\Delta$.} \label{alg:rounding}
\begin{algorithmic}
\State {\bf Input:} Tropical simplex $P_\Delta:=\tconv(v_1, \ldots, v_e)$.
\State {\bf Output:} Minimum enclosing ball, $B_k(Tr_{e-1}(P_\Delta))$.
\State Construct the $h^*-representation(P_\Delta)$.
\State Enumerate the set of pseudo-vertices defining $Tr_{e-1}(P_\Delta)$ using a $double-description$ method~\citep{cdd}.
\State Compute the minimum enclosing ball, $B_k(Tr_{e-1}(P_\Delta))$ using~\eqref{eq:minball}-~\eqref{eq:minball3}.\\
\Return $B_k(Tr_{e-1}(P_\Delta))$.
\end{algorithmic}
\end{algorithm}

 The following example shows how to apply Algorithm~\ref{alg:rounding} to a $P_\Delta$ that is not a polytrope.

\begin{example}[Example~\ref{ex:polytope_vol} cont'd]\label{ex:rounding}
    Consider the tropical polytope from Example~\ref{ex:polytope_vol}.  Note that this is also a tropical simplex, $P_\Delta$. The $B_r(P_\Delta)$ has radius $r=5$ meaning that $Vol(B_r(P_\Delta))=500$.  Proceeding through Algorithm~\ref{alg:rounding}, we get the following set of pseudo-vertices defining the boundary of $Tr_{e-1}(P_\Delta)$:

  \[  \left(\begin{array}{cccccccc}
        0 & 0 &0&0&0&0&0&0 \\
        0 & 0 &0&0&1&1&1&1 \\
        1 & 2 &1&2&2&2&3&3 \\
        4 & 4 &6&7&5&7&5&8 \\
    \end{array} \right).\]

\noindent Computing the new minimum enclosing ball, $B_k(Tr_{e-1}(P_\Delta))$, yields a radius, $r=2$.  This $B_k(Tr_{e-1}(P_\Delta))$ has a volume, $Vol(B_r(Tr_{e-1}(P_\Delta)))=32$.  Figure~\ref{fig:pseudo_simplex} shows the resulting $Tr_{e-1}(P_\Delta)\subset P_\Delta$.
\end{example}

\begin{figure}[H]
    \centering
    \includegraphics[width=0.44\textwidth]{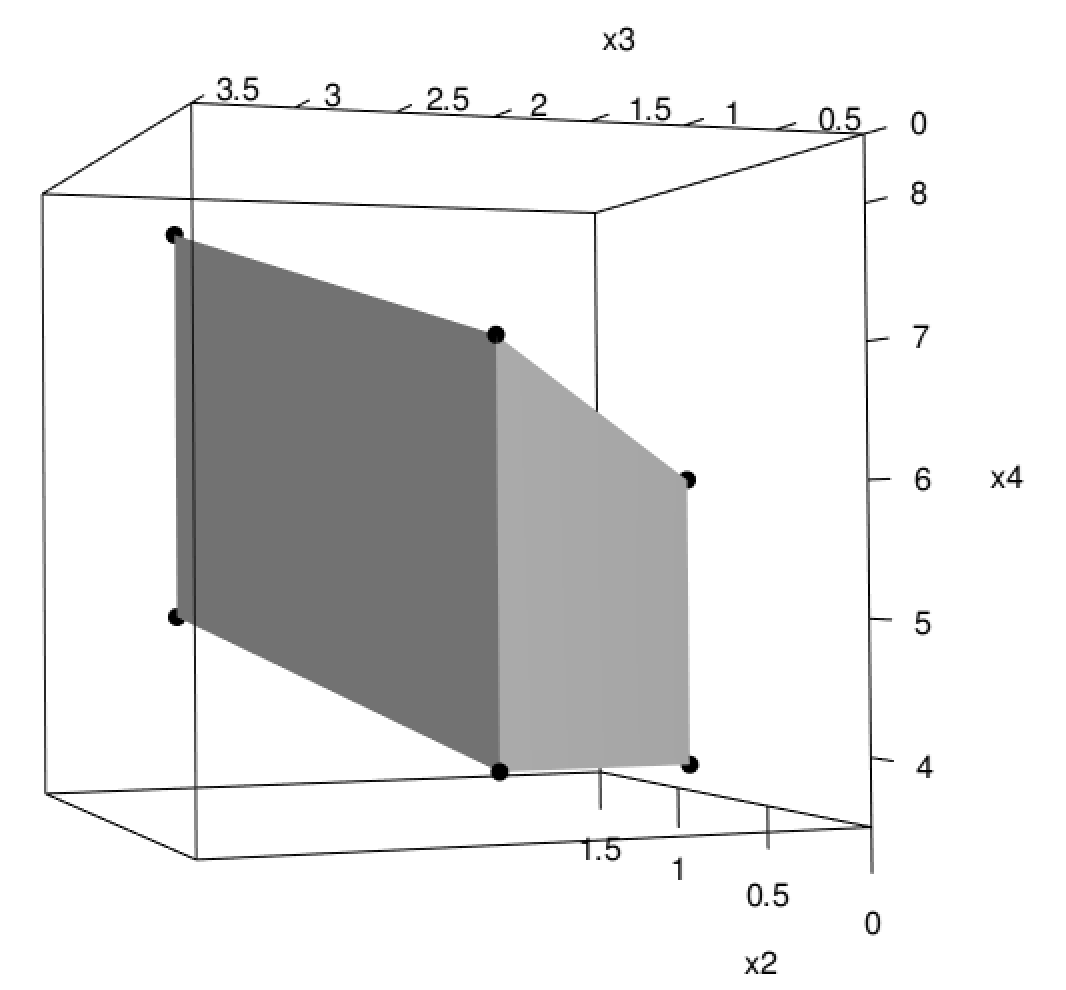}~
    \includegraphics[width=0.44\textwidth]{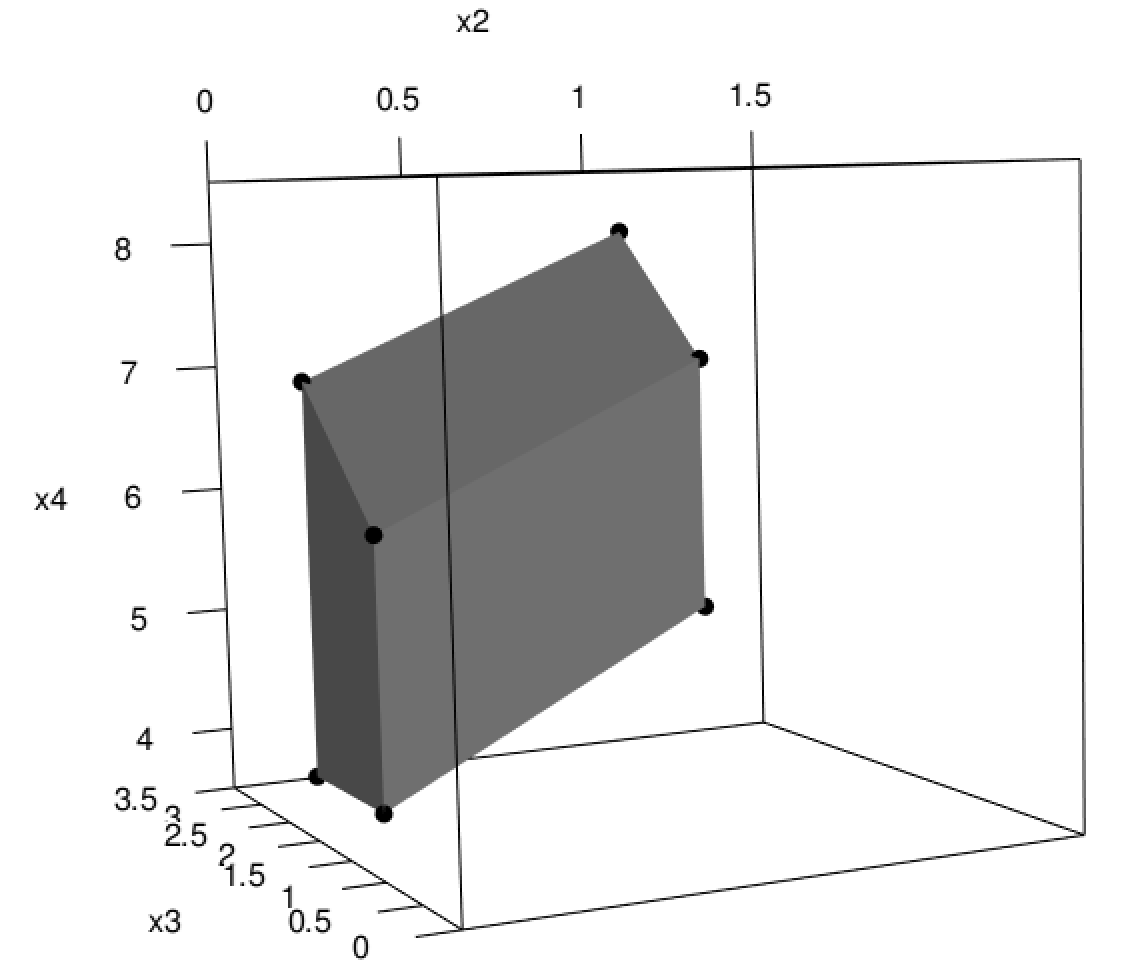}
    \caption{$Tr_{(e-1)}(P)$ defined in Example~\ref{ex:rounding} using  Algorithm~\ref{alg:rounding} with the \texttt{rcdd} package~\citep{rcdd}. Red points represent the pseudo-vertices.}
    \label{fig:pseudo_simplex}
\end{figure}

% Example~\ref{ex:rounding} shows how to round $P$ by removing all {\em (e-2)-tentacles} so only the $Tr_{e-1}(P)$ remains. Since $Tr_{e-1}(P)\subseteq P$, the resulting minimum enclosing ball, $B_k(Tr_{e-1}(P))$, possesses a radius, $k$ such that $k\leq r$, of $B_r(P)$. Since we are sampling from a minimum encompassing ball with $k\leq r$, fewer sample points are required to estimate $Vol(P)$. 

\noindent Table~\ref{tab:Vol_est3} compares the estimate of $Vol(P_\Delta)$ before and after rounding.  

\begin{table}[H]
    \centering
    \begin{tabular}{|l|l|l|}
    \hline
        \multicolumn{1}{|c|}{$I$} &\multicolumn{1}{|c|}{$Vol(P_\Delta)$} &\multicolumn{1}{|c|}{$Vol(Tr_{(e-1)}(P_\Delta))$}\\\hline \hline
        1000 &  1.5&2.752 \\\hline
    10000 & 2.45 & 2.512\\\hline
    100000 &  2.645& 2.517\\\hline
    \end{tabular}
     \caption{Volume estimates of $P_\Delta,Tr_{e-1}(P_\Delta) \in \mathbb{R}^4/\mathbb{R}\mathbf{1}$ with varying number of sampled points, $I$ for Example~\ref{ex:rounding}.}
    \label{tab:Vol_est3}
\end{table}

\section{Sampling from a Union of Tropical Simplices in a Tropical Polytope}\label{sec:uni_samp}

In this section, we focus on uniform sampling from $Tr_{e-1}(P)\subseteq P$ that is not classically convex where $P$ is a tropical polytope.  Our goal is to identify a set of tropical simplices in $\Delta_P$, such that the union of these tropical simplices define $Tr_{e-1}(P)$.  Then, using Algorithm~\ref{alg:HAR_extrapolation4} in~\ref{app:min_plus_HAR}, sample from each tropical simplex according to the proportion of $Tr_{e-1}(P)$ they define. In \cite{Trop_HAR}, the authors show that uniform sampling from a polytrope is possible though do not implement the method.  This can be extended to any tropical simplex, $P_\Delta$, that is not contained within a tropical hyperplane (See Lemma 2 in ~\citep{JoswigKulas+2010+333+352}).   

Unfortunately, for a $Tr_{e-1}(P)\in P$ that is not classically convex, we cannot apply Algorithm~\ref{alg:HAR_extrapolation4} directly due to biased sampling of the boundary of $P$, denoted as $\partial P$. Instead, given a vertex set, $V$ defining $P$, where $|V|=n$ and $n>e$, we must consider all $\binom{n}{e}$ combinations of $e$ vertices in $V$.  Each combination defines a tropical simplex, $P_\Delta^i\in \Delta_P$ where $\bigcup_{i=1}^{\binom{n}{e}} P_\Delta^i= P$, which also defines $Tr_{(e-1)}(P)$.

Once we enumerate each $P_\Delta^i \in \Delta_P$, we must find a subset of tropical simplices, $\Delta_P'\subset \Delta_P$, such that $\bigcup_{k\in [|\Delta_P'|]} P_\Delta^k=P' \subseteq P$ where $Tr_{(e-1)}(P) \subset P'$ and $0 \leq dim(P_\Delta^i \cap P_\Delta^j)\leq e-2$ for all $P_\Delta^i,P_\Delta^j \subseteq \Delta_P'$. To find $\Delta_P'$, we can use Algorithm~\ref{alg:HAR_extrapolation4} to sample from $B_r(P)$ and identify those points of the entire sample, $X$, that fall inside of $P$. The sampled points that fall in $P$ we call $X_P$. We then determine the proportion, $p_i$, of points in $X_P$ that fall in each $P_\Delta^i \in \Delta_P$.  Finally, we say $\Delta_P'=\left\{P_\Delta^i,\ldots,P_\Delta^k\right\}$ where $i,k\in\left[\left|\binom{n}{e}\right|\right]$ and $\sum\limits_{j\in\{i\ldots, k\}} p_j=1$. 

\begin{algorithm}[H]
\caption{Identifying $\Delta_P'\subset \Delta_P$ such that $0\leq dim(P_\Delta^i\cap P_\Delta^j)\leq e-2$ $\forall\; P_\Delta^i,P_\Delta^j \subset \Delta_P'$ where $\bigcup_{k\in [|\Delta_P'|]} P_\Delta^k=P'$ such that $Tr_{(e-1)}(P) \subset P'$.} \label{alg:subset}
\begin{algorithmic}
\State {\bf Input:} Tropical polytope $P:=\tconv(v_1, \ldots, v_e)$; starting point, $x_0$; sample size, $I$.
\State {\bf Output:} A set of tropical simplices, $\Delta_P'\in \Delta_P$ and $\mathbf{p}=\{p_i\ldots p_k\}$ where $i,k\in\left[\left|\binom{|V|}{e}\right|\right]$.
\State Sample $I$ points from $B_r(P)$ using Algorithm~\ref{alg:HAR_extrapolation4} and call this set $X$.
\State Let $X_P \subset X$ where $X_P$ is the subset of points that fall inside $P$.
\State Enumerate each tropical simplex $P_\Delta^i \subset P$.
\State For each $P_\Delta^i \subset P$ calculate $\frac{|X_{P}\in P_\Delta^i|}{|X_P|}=p_i$. 
\State Let $\Delta_P'=\left\{P_\Delta^i,\ldots,P_\Delta^k\right\}$ where $P'=\bigcup_{j\in\{i,\ldots,k\}} P_\Delta^j$ such that $\sum\limits_{j\in\{i\ldots, k\}} p_j=1$. 
\\
\Return $\Delta_P'$ and $\mathbf{p}=\{p_i\ldots p_k\}$ where $i,k\in[\left|\binom{n}{e}\right|]$.
\end{algorithmic}
\end{algorithm}

After identifying $\Delta_P'$ and $\mathbf{p}$, we then employ Algorithm~\ref{alg:HAR_extrapolation4} again and sample each $P_\Delta^i\in\Delta_P'$ according to its proportion $p_i \in \mathbf{P}$ as shown in Algorithm~\ref{alg:unisample}.

\begin{algorithm}[H]
\caption{Uniform Sampling $Tr_{(e-1)}(P)\in P$, where $P\in \mathbb{R}^e/\mathbb{R}\mathbf{1}$ and $Tr_{e-1}(P)$ is not classically convex.} \label{alg:unisample}
\begin{algorithmic}
\State {\bf Input:} A set of tropical simplices $\Delta_P'=\left\{P_\Delta^i,\ldots,P_\Delta^k\right\}$ and $\mathbf{p}=\{p_i,\ldots,p_k\}$ where $i,k\in\left[\left|\binom{n}{e}\right| \right]$ taken from Algorithm~\ref{alg:subset}; starting point, $y_0$; and a sample size $J$.
\State {\bf Output:} A set, $Y$, of points sampled uniformly from $Tr_{(e-1)}(P)$.
\For {$j=0\ldots J-1$}
\State Randomly select $P_\Delta^i\subset \Delta_P'$ according to proportions defined in $\mathbf p$.
\State Sample $y_{j+1}$ from $P_\Delta^i$ using Algorithm~\ref{alg:HAR_extrapolation4}.
\EndFor\\
\Return $Y$.
\end{algorithmic}
\end{algorithm}

As an example, we refer back to Figure~\ref{fig:tpolex}.  Note that the vertex set, $V$, defining $P\in\mathbb{R}^3/\mathbb{R}\mathbf{1}$ has a cardinality, $|V|=4$.  This results in $\binom{4}{3}$ tropical simplices to examine (See Figure~\ref{fig:tpolex1}). In this case, either the pair of tropical simplices on the left of Figure~\ref{fig:tpolex1} or the right represent sets of tropical simplices that only intersect at their boundary but still define $Tr_{e-1}(P)$. We will consider the pair of tropical simplices on the right in the following example.

\begin{example}\label{ex:uni_samp1}
    Consider the tropical polytope $P=\{(0,-2,3),(0,-2,5), (0,2,2),(0,1,0)\}$.  The tropical polytope, $P$, may be decomposed into the two tropical simplices, $P_\Delta^1=\{(0,-2,5),(0,-2,3),$\\$(0,1,0)\})$ with $p_1=0.62$ and $P_\Delta^2=\{(0,-2,5),(0,1,0),(0,2,2)\})$ with $p_2=0.38$. Using Algorithm~\ref{alg:unisample} and setting the sample size, $I=2000$, we sample $P_\Delta^1$ and $P_\Delta^2$ according to the proportions $p_1$ and $p_2$. Figure~\ref{fig:uni_results1} shows the results for this example.

    \begin{figure}[H]
    \centering
\includegraphics[width=0.55\textwidth]{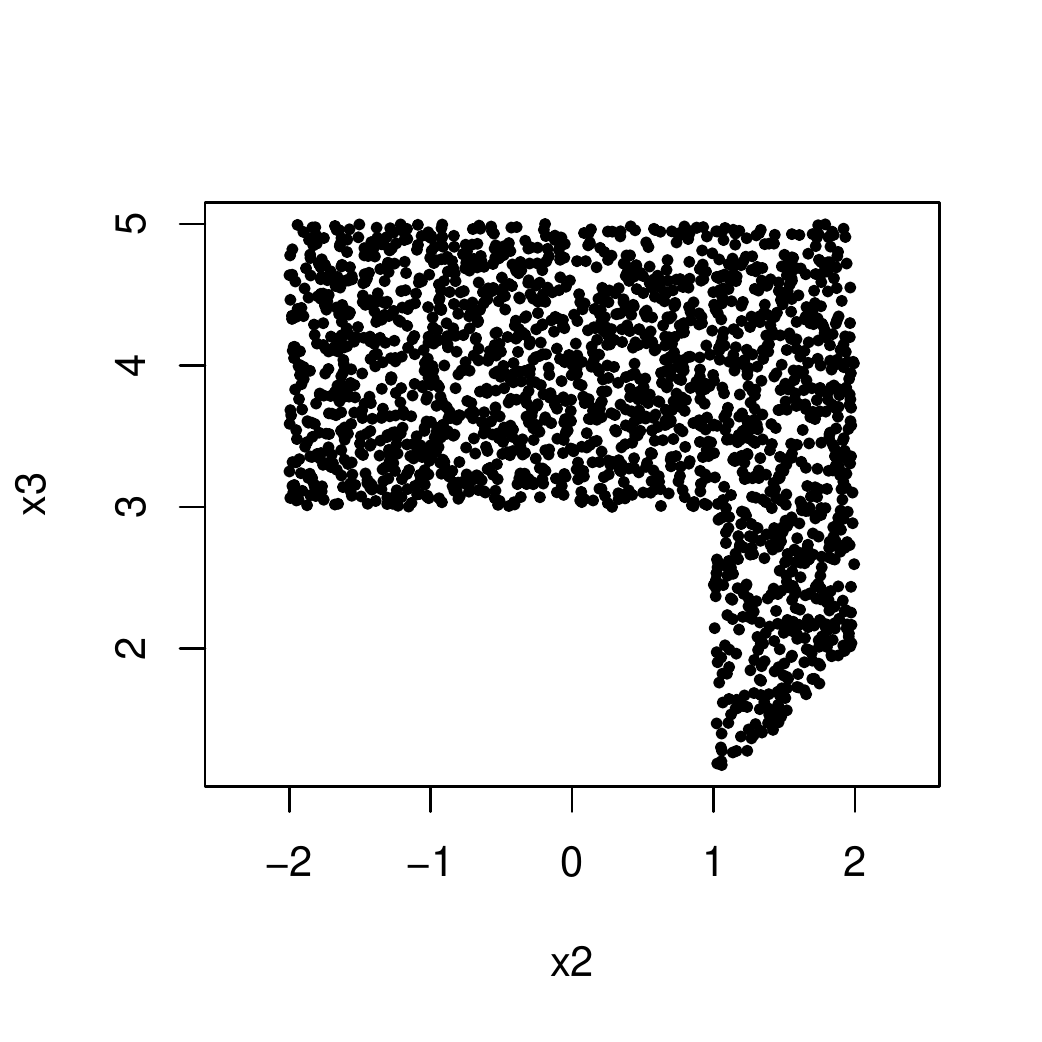}~
    \caption{Results for Example~\ref{ex:uni_samp1}.}
    \label{fig:uni_results1}
\end{figure}
\end{example}

\begin{example}\label{ex:uni_samp2}
    Consider the tropical polytope $P=\{(0,0,2,0),
            (0,0,0,0),
            (0,4,-10,0),
            (0,3,-3,5),$\\$
            (0,4,2,10)\}$. Here $P$ may be decomposed into five tropical simplices.  While there is more than one combination of tropical simplices that define $Tr_{e-1}(P)$, we identify $P_\Delta^1=\{(0,0,2,0),(0,0,0,0),$\\$(0,4,-10,0),(0,4,2,10)\}$ with $p_1=0.97$ and $P_\Delta^2=\{(0,0,0,0),(0,4,-10,0),(0,3,-3,5),(0,4,2,10)\}$ with $p_2=0.03$. Using a sample size of, $I=3000$, we sample $P_\Delta^1$ and $P_\Delta^2$. Figure~\ref{fig:uni_results2} shows the results for this example.

 \begin{figure}[H]
    \centering
    \includegraphics[width=0.44\textwidth]{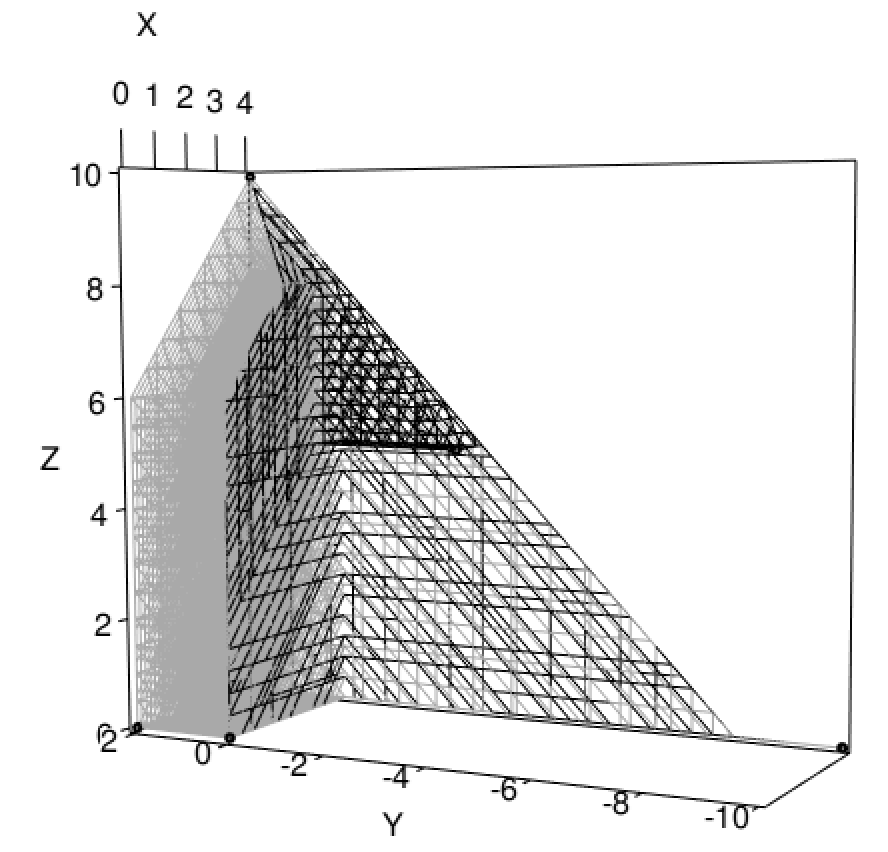}~
    \includegraphics[width=0.44\textwidth]{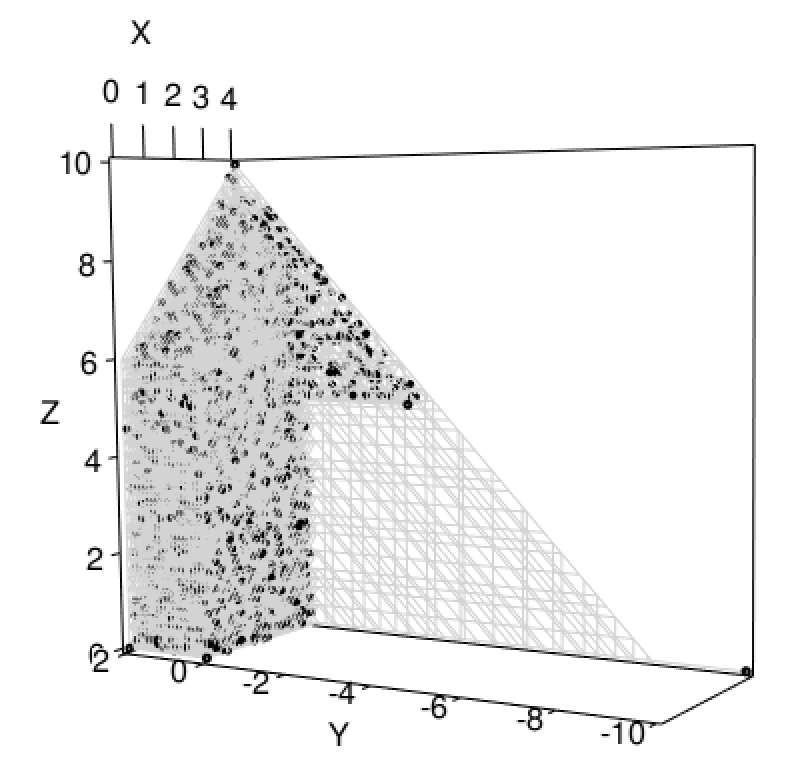}
    \caption{Results for Example~\ref{ex:uni_samp2}. Black and gray colors differentiate between each $P_\Delta^i$ where $\bigcup_{i=1}^2 P_\Delta^i=P'$ (left). Black points represent sampling results from $Tr_{e-1}(P)$ (right).}
    \label{fig:uni_results2}
\end{figure}
\end{example}

\section{Conclusion}

In this paper we show that computing the maximum inscribed and minimum enclosing tropical balls of a tropical polytope can be formulated as linear programming problems and then we applied these tropical balls of a tropical polytope to estimate the volume of and uniform sampling from the tropical polytope over $\mathbb{R}^e/\mathbb{R}{\bf 1}$. 

  Gaubert and Marie MacCaig showed that even estimating the volume of a tropical polytope is NP-hard~\cite{Gaubert}.  The method described here, despite being easily implemented, does not change the difficulty of estimating the volume of a tropical polytope.    In addition to the difficulty of volume estimation, Murty and Oskoorouchi showed that for a given classical polytope $P$ with the vertex representation, i.e, the polytope as a convex hull of a finitely many set of vertices, and for a given point $x_0 \in P$, computing a radius of the largest ball inscribed in $P$ with $x_0$ as the center is NP-hard \cite{Murty}. Therefore, this leads a following problem:
\begin{problem}
What is the time complexity of computing the radius and its center of a maximum inscribed tropical ball given a tropical polytope as a tropical convex set of a finitely many points in $\mathbb{R}^e/\mathbb{R}{\bf 1}$?
\end{problem}
Similary, Shenmaier proved that for a given classical polytope $P$ with the vertex presentation, finding a center and computing a radius of the smallest ball enclosing a polytope $P$ is NP-hard \cite{Shenmaier}.  Thus we have the following open problem:
\begin{problem}
What is the time complexity of computing the radius and its center of a minimum enclosing tropical ball given a tropical polytope as a tropical convex set of a finitely many points in $\mathbb{R}^e/\mathbb{R}{\bf 1}$?
\end{problem}

\appendix
\section{Vertex Hit-and-Run Sampling with Extrapolation}\label{app:min_plus_HAR}

In this appendix we describe the MCMC HAR sampler used in Algorithms~\ref{alg:Vol_Est},~\ref{alg:subset} and~\ref{alg:unisample}. The sampler is defined as a \textit{Vertex HAR Sampler with Extrapolation} algorithm.  This type of sampler samples uniformly from the $(e-1)$-trunk (see Definition~\ref{def:trunk}) of tropical simplices and was originally described in~\citep{Trop_HAR}.  Specifically, in a tropical simplex $P_\Delta:=\tconv(v^1, \ldots, v^s)$ in $\mathbb{R}^e/\mathbb{R}{\bf 1}$ with minimal vertex set $V'$, an ``extrapolated'' point from $v^i \in V'$ through a point $x$ is defined by the projection to the other vertices or $V'^{-i}=\{v^1, \ldots, v^{i-1}, v^{i+1}, \ldots, v^e\}$ as
\begin{equation}\label{eq:tropproj2}
\pi_{V'^{-i}} (x) := % \sum_{l \neq i} \lambda_l \odot w_l,
\bigoplus_{\substack{l=1 \\ l\neq i}}^e \lambda_l \odot v^l~ ~ {\rm where} ~ ~ \lambda_l \!=\! {\rm min}(x-v^l). 
\end{equation}

The authors show that sampling along the line segment $\Gamma_{\pi_{V'^{-i}} (x),v^i}$, is done so uniformly.  However, results showed that sampling was biased because, in general, $\Gamma_{\pi_{V'^{-i}} (x),v^i}$ intersects the boundary of $P_\Delta$, or $\partial P_\Delta$, prior to reaching $v^i$ leading to oversampling of $\partial P_\Delta$.    

To avoid biased sampling of $\partial P_\Delta$, we must detect when $\Gamma_{\pi_{V^{-i}} (x),v^i}$, intersects $\partial P_\Delta$ and ignore the portion of $\Gamma_{\pi_{V^{-i}} (x),v^i}$ that continues on $\partial P_\Delta$ past the initial intersection.  This ensures that the line segment from which we sample lies in the interior of $P$ preventing oversampling $\partial P_\Delta$. While one of the end points of this line segment will be $\pi_{V^{-i}} (x)$, the other will be one of the {\em bend points} of $\Gamma_{\pi_{V^{-i}} (x),v^i}$.

\begin{definition}[Bend point from~\citep{MS}]
Given two points $v^1, \, v^2$, a tropical line segment between $v^1, \, v^2$ denoted as $\Gamma_{v^1,v^2}$, consists of the concatenation of at most $e-1$ Euclidean line segments.  The collection of points $B$, defining each Euclidean line segment are called the {\em bend points} of $\Gamma_{v^1,v^2}$.  Including $v^1$ and $v^2$, $\Gamma_{v^1,v^2}$ consists of at most $e$ bend points. To compute the set $B$ defining $\Gamma_{v^1,v^2}$ we have the following

\begin{equation}\label{eq:troline}
B=\left\{
\begin{array}{ccl}
\!\!\!(v_e \!-\! u_e) \odot u \oplus v \!\!\!\! &=& \!\!\! v\\ % = (v_1, v_2, v_3, \ldots , v_{d-1},  v_e)\\
\!\!\!(v_{e-1} \!-\! u_{e-1}) \odot u \oplus v \!\!\!\!&=&\!\!\! (v_1, v_2, v_3, \ldots , v_{e-1}, v_{e-1} - u_{e-1} + u_e)\\
&\vdots& \\
\!\!\!(v_2 \!-\! u_2) \odot u \oplus v \!\!\!\!&=&\!\!\! (v_1, v_2, v_2 - u_2 + u_3, \ldots , v_2 - u_2 + u_e)\\
\!\!\! (v_1 \!-\! u_1) \odot u \oplus v \!\!\!\!&=&\!\!\! u .\\
%&=& (v_1, v_1 - u_1 + u_2, v_1 - u_1 + u_3, \ldots , v_1 - u_1 + u_e).\\
\end{array}\right\}.    
\end{equation}.
\end{definition}

Therefore, using the fact that the boundary of the $Tr_{(e-1)}(P_\Delta)$ in a tropical simplex, $P_\Delta$ is the intersection of closed sectors in $\mathcal{A}(V')$, we can detect the intersection of $\Gamma_{\pi_{V'^{-i}} (x),v^i}$, with $\partial P_\Delta$ by successively evaluating the distance from each bend point, $b^j \in B$, of $\Gamma_{\pi_{V'^{-i}} (x),v^i}$ to each $H_{\omega^i}^{min}$.  We compute this distance using Equation~\eqref{eq:hyp_dist_zero} 

\begin{equation}\label{eq:hyp_dist_zero}
    d_{tr}(x,H_{\omega^i}^{min})=d_{tr}(x+\omega,H_{\mathbf{0}}^{min}),
\end{equation}

\noindent which is shown to be true in~\citep{gartner2008tropical}.
If $d_{tr}(b^j,H_{\omega^i}^{min})=0$ for any $\omega^i$, then $\partial P_\Delta$ has been detected.  We then can truncate $\Gamma_{\pi_{V'^{-i}} (x),v^i}$ to $\Gamma_{\pi_{V'^{-i}} (x),b^j}$ and sample from $\Gamma_{\pi_{V'^{-i}} (x),b^j}$ using Algorithm 1 from~\citep{Trop_HAR}. If $\partial P_\Delta$ is not detected for any $b_j\in B$, then we use Algorithm 1 from~\citep{Trop_HAR} to sample along the entirety of $\Gamma_{\pi_{V'^{-i}} (x),v^i}$. This leads to Algorithm~\ref{alg:HAR_extrapolation4}, which randomly subsets $V'$ into subsets, $V'^{-i}$ and $\{v^i\}$ and only samples between $\pi_{V'^{-i}}(x)$ and the first intersection of $\Gamma_{\pi_{V'^{-i}} (x),v^i}$ with $\partial P_\Delta$.

\begin{algorithm}[H]
\caption{Vertex HAR extrapolation sampling from $P$ with truncation using hyperplane distance calculation.} \label{alg:HAR_extrapolation4}
\begin{algorithmic}
\State {\bf Input:} $P_\Delta:=\tconv(v^1, \ldots, v^s)$, an initial point $x_0 \in P_\Delta$ and maximum iteration $I \geq 1$.
\State {\bf Output:} A random point $x \in P_\Delta$.
\State Let, $\Omega$, be the the set of normal vectors, $\omega^i$, defining $H^{\min}_{\omega^i}$ at each $v^i \in V'$.
\For{$k= 0, \ldots , I-1$,}
\State Randomly select a non-empty set ${V'}_{k}^{-i}$ with $v^i \notin {V'}_{k}^{-i}$.
\State Construct $\Gamma_{\pi_{{V'}_{k}^{-i}}(x_k),v^i}$.
\State Order the set of bend points, $B$, defining line segment $\Gamma_{\pi_{V_k^{-i}}(x_k),v^i}$, excluding the end points.
\State Set $j=1$
\While {$j \leq |B|$}
\State Calculate $d_{tr}(b^j,H^{\min}_{\omega^i})$ for each $\omega^i\in \Omega$ where $b^j\in B$ using Equation~\eqref{eq:hyp_dist_zero}.
\If{$d_{tr}(b^j,H^{\min}_{\omega^i})=0$ for some $\omega^i \in \Omega$}
\State Let $\Gamma^k=\Gamma_{\pi_{V_k^{-i}}(x_k),b^j}$.
\Else 
\If{$j<|B|$}
\State Set $j=j+1$
\Else
\State Let $\Gamma^k=\Gamma_{\pi_{V_k^{-i}}(x_k),v^i}$
\EndIf
\EndIf
\EndWhile
\State Generate a random point $x_{k+1}$ from the tropical line segment $\Gamma^k$ using Algorithm 1 from~\citep{Trop_HAR}.
\EndFor \\
\Return $x := x_{I}$.
\end{algorithmic}
\end{algorithm}

\begin{exam}\label{ex:app_ex}
    Let $P_\Delta:=\tconv(\{(0,-1,1),(0,0,0),(0,1,-1)\})$. Figure~\ref{fig:ext_ex} shows the results from sampling 2000 points from $P_\Delta$ using Algorithm~\ref{alg:HAR_extrapolation4}. 
\end{exam}

\begin{figure}[H]
    \centering
    \includegraphics[width=0.4\textwidth]{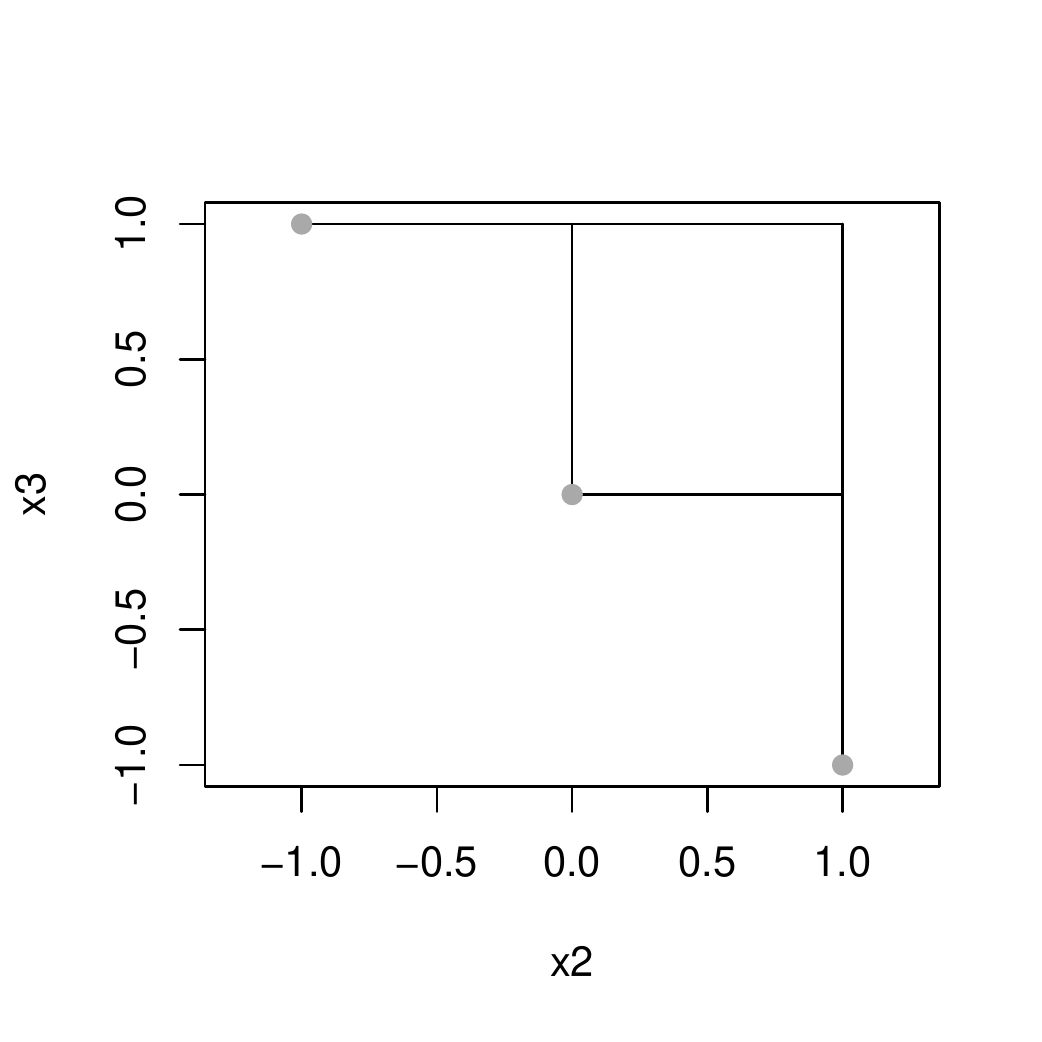}~
    \includegraphics[width=0.4\textwidth]{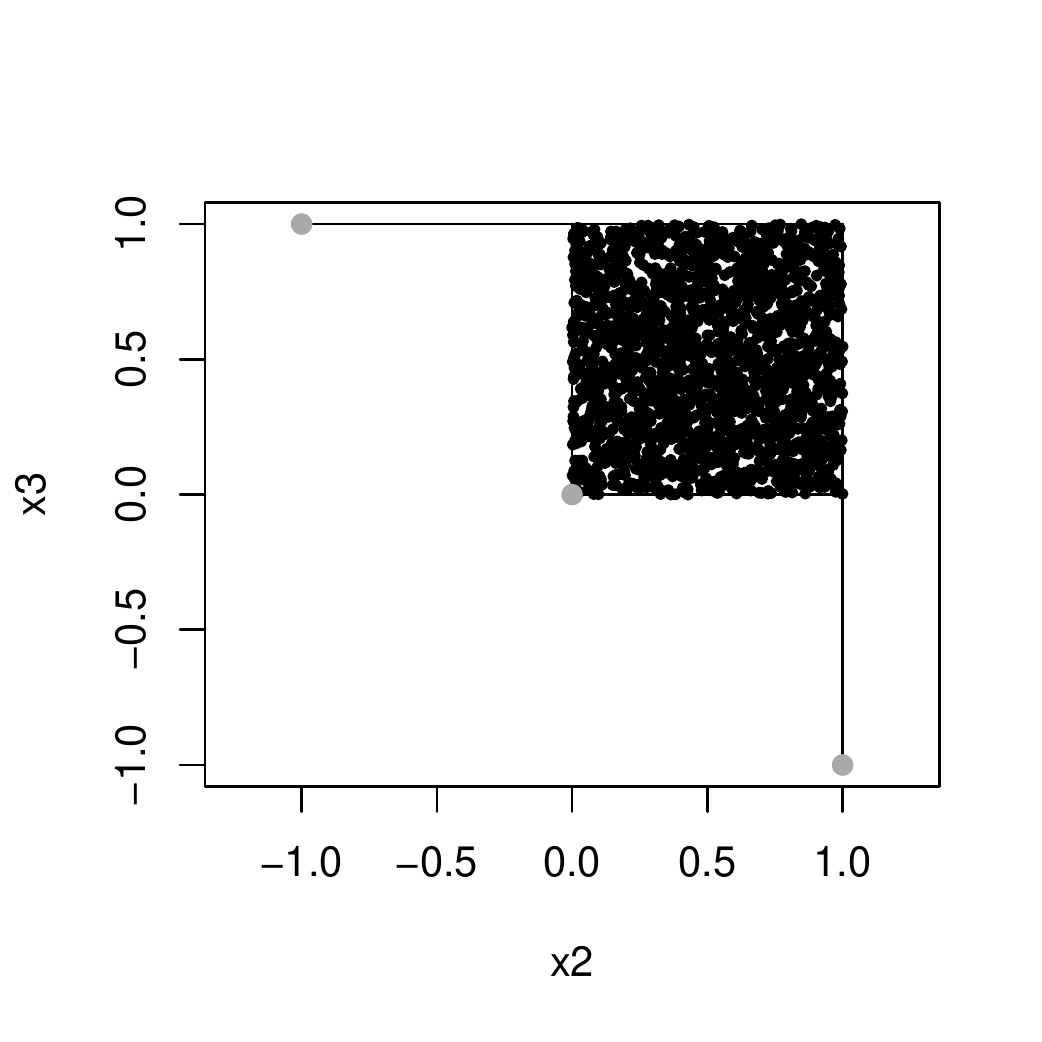}
    \caption{Results for Example A.\ref{ex:app_ex} after sampling 2000 points using Algorithm~\ref{alg:HAR_extrapolation4} for $P_\Delta\in \mathbb{R}^3/\mathbb{R}\mathbf{1}$.}
    \label{fig:ext_ex}
\end{figure}

\section*{Acknowledgement}

The authors thank Profs.~Michael Joswig and Ngoc Tran for useful conversations and discussions. RY and DB are partially supported from NSF DMS 1916037.
KM is partially supported by JSPS KAKENHI Grant Numbers	JP18K11485, JP22K19816, JP22H02364.

\bibliographystyle{plain}  
\bibliography{refs}  

\end{document}